\documentclass[bezier,amstex, oneside,reqno]{amsart}
\usepackage{amsmath, amssymb, amsthm, verbatim, euscript, amscd, textcomp, slashbox, tabularx, booktabs,arydshln}
\usepackage[dvips]{graphicx}
 \usepackage[all,cmtip]{xy}
 \usepackage[fleqn]{mathtools}
\usepackage{epsfig}
\usepackage{color}

\def\ie{\emph{i.e., }}
\def\eg{\emph{e.g., }}

\def\S{\mathbb S}
\def\DD{\mathbb D}
\def\R{\mathbb R}
\def\Z{\mathbb Z}

\def\W{\mathcal W}

\def\cP{{P}}
\def\ccP{{\EuScript P}}

\def\T{\mathbb T}

\def\N{\EuScript N}
\numberwithin{equation}{section}

\newtheorem{theorem}{Theorem}[section]
\newtheorem{q}[theorem]{Question}
\newtheorem{conj}[theorem]{Conjecture}
\newtheorem{prop}[theorem]{Proposition}
\newtheorem{add}[theorem]{Addendum}
\newtheorem{claim}[theorem]{Claim}

 \theoremstyle{remark}
\newtheorem{remark}[theorem]{Remark}
\newtheorem{cor}[theorem]{Corollary}
\newtheorem{lemma}[theorem]{Lemma}

\theoremstyle{remark}

\renewcommand\labelenumi{\theenumi.}

\makeatletter
\renewcommand*\env@matrix[1][*\c@MaxMatrixCols c]{%
  \hskip -\arraycolsep
  \let\@ifnextchar\new@ifnextchar
  \array{#1}}
\makeatother

\begin{document}
\author{F. Thomas Farrell$^\ast$ and Andrey Gogolev$^{\ast\ast}$}
\title[Bundles with hyperbolic dynamics]{On bundles that admit fiberwise hyperbolic
dynamics}
\thanks{$^\ast$The first author was partially supported by NSF grant DMS-1206622.\\
$^{\ast\ast}$The second author was partially supported by NSF grants DMS-1204943, 1266282. He also would
like to acknowledge the support provided by Dean's Research Semester Award at SUNY Binghamton.}
\begin{abstract}
This paper is devoted to rigidity of smooth bundles which are equipped with fiberwise geometric
or dynamical structure. We show that the fiberwise associated sphere bundle to a bundle whose leaves are equipped with (continuously varying) metrics of negative curvature is a topologically trivial bundle when either the base space is simply connected or, more generally, when the bundle is fiber homotopically trivial. We present two very different proofs of this result:  a geometric proof and a dynamical proof. We also establish a number of rigidity results for bundles which are equipped with fiberwise Anosov dynamical systems. Finally, we present several examples which show that our results are sharp in certain ways or illustrate necessity of various assumptions.
\end{abstract}
\date{}
 \maketitle

\section{Negatively curved bundles}
\subsection{The results}
Recall that a {\it smooth bundle} is a fiber bundle $M\to E\stackrel{p}{\to} X$ whose fiber $M$ is a connected closed manifold and whose structure group is the diffeomorphism group $\textup{Diff}(M)$. A {\it negatively curved bundle}  is a smooth bundle  $M\to E\stackrel{p}{\to} X$ whose fibers $M_x=p^{-1}(x)$, $x\in X$, are equipped with Riemannian metrics $g_x$, $x\in X$, of negative curvature. Furthermore, we require that the Riemannian metrics $g_x$ vary continuously with $x$ in the $C^\infty$ topology. 

We say that a smooth bundle $M\to E\stackrel{p}{\to} X$ is {\it topologically trivial} if there exists a continuous map $r\colon E\to M$ such that the restriction $r|_{M_x}\colon M_x\to M$, is a homeomorphism for each $x\in X$. Note that if $M\to E\stackrel{p}{\to} X$ is topologically trivial then $(p,r)\colon E\to X\times M$ is a fiber preserving homeomorphism.
\begin{conj}[Farrell-Ontaneda]
 Let $X$ be a compact simply connected manifold or a simply connected finite simplicial complex and let $M\to E\stackrel{p}{\to} X$ be a negatively curved bundle. Then the bundle $p\colon E\to X$ is, in fact, topologically trivial.
\end{conj}
\begin{remark} Techniques developed in~\cite{FO1, FO2, FO3} can be used to verify this conjecture in certain special cases (when $\dim M\gg\dim X$ and $X$ is a sphere). However the general conjecture above seems to be out of reach to these topological techniques. 
\end{remark}
In this paper classical dynamical systems techniques are used to verify related results in full generality which we proceed to describe. Given a smooth closed manifold $M$ we define its {\it tangent sphere bundle} $SM$ as the quotient of $TM\backslash M$ by identifying two non-zero tangent vectors $u$ and $v$ based at a common point $x\in M$ if and only if there exists a positive number $\lambda$ such that $u=\lambda v$. Similarly given a smooth bundle $M\to E\stackrel{p}{\to} X$ we define its {\it (fiberwise) associated  sphere bundle} $SM\to SE\stackrel{S(p)}{\to} X$ as the quotient of the space of non-zero vectors tangent to the fibers of $p\colon E\to X$. Given a (class of a) vector $v\in SE$ with a base point $x\in E$ we define $S(p)(v)=p(x)$. Then it is clear that the fibers of the associated sphere bundle are diffeomorphic to $SM$. Also note that if the fibers of a smooth bundle are equipped with continuously varying Riemannian metrics then the total space of the associated sphere bundle can be realized as the space of unit vectors tangent to the fibers of the original bundle.  
\begin{theorem}\label{thm_main}
 Let $p\colon E\to X$ be a smooth negatively curved bundle whose base space $X$ is a simply connected closed manifold or a simply connected finite simplicial complex. Then its (fiberwise) associated sphere bundle $S(p)\colon SE\to X$ is topologically trivial.
\end{theorem}
We say that a smooth bundle $M\to E\stackrel{p}{\to} X$ is {\it fiber homotopically trivial} if there exists a continuous map $q\colon E\to M$ such that the restriction $q|_{M_x}\colon M_x\to M$ is a homotopy equivalence for each $x\in X$. 
\begin{prop}\label{prop_trivialization}
 If $X$ is a simply connected closed manifold or a simply connected finite simplicial complex and $M\to E\stackrel{p}{\to} X$ is negatively curved bundle then it is fiber homotopically trivial.
\end{prop}
Because of the above proposition Theorem~\ref{thm_main} is a consequence of the following more general result.
\begin{theorem}\label{thm_main_general}
  Let $X$ be a closed manifold or a finite simplicial complex and $p\colon E\to X$ be a smooth negatively curved, fiber homotopically trivial bundle. Then its (fiberwise) associated sphere bundle $S(p)\colon SE\to X$ is topologically trivial.
\end{theorem}

It was shown in~\cite{FO2} that there are ``lots" of smooth bundles $E\to\mathbb S^k$, $(k>1)$ whose abstract fiber $M$ supports a negatively curved Riemannian metric, but which are not negatively curved. So Theorem~\ref{thm_main} does not apply to these bundles. Hence it is a priori conceivable that the conclusion of Theorem~\ref{thm_main} holds under the weaker assumption that the abstract fiber $M$ supports a negatively curved Riemannian metric. However the next result shows that this is not the case.
\begin{theorem}\label{thm_2}
For each prime number $p>2$, there exists a smooth bundle $M\to E\to \S^{2p-3}$ such that
\begin{enumerate}
 \item 
the (abstract) fiber $M$ is real hyperbolic, \ie $M$ supports a Riemannian metric of constant negative sectional curvature;
\item the (fiberwise) associated sphere bundle $SM\to SE\to \S^{2p-3}$ is not topologically trivial.
\end{enumerate}
\end{theorem}
We will prove Theorem~\ref{thm_2} in Section~\ref{sec_thm_2}. Now we proceed to give concise proofs of Proposition~\ref{prop_trivialization} and Theorem~\ref{thm_main_general}.

\subsection{Proofs}
\begin{proof}[Proof of Proposition~\ref{prop_trivialization}]
 We first consider the case where $X$ as a finite simplicial complex. Let $X^1\subset X$ be the 1-skeleton of $X$. First note that the inclusion map $\sigma\colon X^1\to X$ is homotopic to a constant map with value, say, $x_0\in X$. By the covering homotopy theorem, this homotopy is covered by a bundle map homotopy starting with the inclusion map $p^{-1}(X^1)\subset E$ and ending with a continuous map $q^1\colon p^{-1}(X^1)\to M_{x_0}$ which is a homotopy equivalence when restricted to each fiber.
 
 Denote by $G(M)$ the space of self homotopy equivalences of $M$. Then
 \begin{equation}
 \label{eq_homotopy_homotopy}
 \pi_n(G(M))=
 \begin{cases}
  Center(\pi_1(M)), & \mbox{if}\,\, n=1\\
  0, & \mbox{if}\,\, n\ge 2,
 \end{cases}
 \end{equation}
since $M$ is aspherical (see~\cite{Gott}). Note that center of $\pi_1(M)$ is trivial because $M$ is negatively curved. Hence $\pi_n(G(M))$ is trivial for $n\ge 1$ and we can inductively extend map $q^1$ to succesive skeletons of $X$ until we obtain a fiber homotopy trivialization $q$.

We now consider the case where $X$ is a closed manifold. If $X$ is smooth, then we are done by the first case since every smooth manifold can be triangulated. But there are many examples of closed topological manifolds that cannot be triangulated. However, every closed manifold is an ANR and, hence, is homotopy equivalent to finite simplicial complex $K$ (see~\cite{W}). Let $r\colon K\to X$ be a homotopy equivalence and $\sigma\colon K\to X$ be its homotopy inverse.  Also let
$$
(\xi_K)\;\;\;\;\;\;\;\; M\to \mathcal E\to K\;\;\;\;\;\;\;\;
$$
denote the pullback bundle along $r\colon K\to X$ of the given bundle
$$
(\xi_X)\;\;\;\;\;\;\;\; M\to E\stackrel{p}{\to} X\;\;\;\;\;\;\;\;
$$
Then bundle $(\xi_K)$ is negatively curved and, hence, is fiber homotopically trivial by the argument in the first case. 

But if $\xi_K$ is fiber homotopically trivial, then so is $\xi_X$ since
$$
\xi_X=id^*(\xi_X)=\sigma^*(r^*(\xi_X))=\sigma^*(\xi_K).
$$
\end{proof}
\begin{proof}[Proof of Theorem~\ref{thm_main_general}]
 Denote by $q\colon E\to M$ the fiber homotopy trivialization. Choose a base point $x_0\in X$. We identify the abstract fiber $M$ of the bundle with the fiber $M_{x_0}$ over $x_0$. We choose this identification so that $q|_M\colon M\to M$ induces the identity map on $\pi_1(M)$. Therefore $q_\#\colon \pi_1(E)\to\pi_1(M)$ splits
 the short exact sequence
 $$
 0\to\pi_1(M)\to\pi_1(E)\to\pi_1(X)\to 0.
 $$
Hence $\pi_1(E)\simeq\pi_1(X)\times\pi_1(M)$. 

Let $\rho\colon\tilde E\to E$ be a covering map such that $\rho_\#\pi_1(\tilde E)=\pi_1(X)\subset\pi_1(E)$. Let $\tilde p=p\circ\rho$. Note that $\rho_\#$ intertwines the long exact sequences for $p$ and $\tilde p$:
$$
\begin{xymatrix}
{
\ldots\ar[r] & \pi_2(X)\ar[d]_{id}\ar[r] & \pi_1(\tilde M)\ar[d]_{\rho_\#}\ar[r] & \pi_1(\tilde E)\ar[d]_{\rho_\#}\ar[r]^{\textup{isom.}} &\pi_1(X)\ar[d]_{id}\ar[r] & \pi_0(\tilde M)\ar[r] & 0\\
\ldots\ar[r] & \pi_2(X) \ar[r]^0 &  \pi_1(M)\ar[r] & \pi_1(E)\ar[r] & \pi_1(X)\ar[r] & 0 &
}
\end{xymatrix}
$$
We read from the above diagram that the fiber $\tilde M$ is a connected simply connected manifold.

Recall that $p\colon E\to X$ is a negatively curved bundle. The negatively curved Riemannian metrics on the fibers of $p\colon E\to X$ lift to negatively curved Riemannian metrics on the fibers of $\tilde p\colon\tilde E\to X$. For each fiber $\tilde M_x$, $x\in X$, consider the ``sphere at infinity" $\tilde M_x(\infty)$, which is defined as the set of equivalence classes of asymptotic geodesic rays in $\tilde M_x$. Given a point $y\in\tilde M_x$, consider the map $r_y\colon S_y\tilde M_x\to\tilde M_x(\infty)$ defined by $v\mapsto [\gamma_v]$, where $\gamma_v$ is the geodesic ray with $\gamma_v(0)=y$, $\dot\gamma_v(0)=v$, and $[\gamma_v]$ is the equivalence class of $\gamma_v$. Map $r_y$ is a bijection that induces sphere topology on $\tilde M_x(\infty)$. (This topology does not depend on the choice of $y\in \tilde M_x$.)

Let $\tilde E(\infty)=\sqcup_{x\in X}\tilde M_x(\infty)$. Charts for $\tilde p\colon\tilde E\to X$ induce charts for the bundle $\tilde p_\infty\colon \tilde E(\infty)\to X$ whose fiber is the sphere at infinity $\tilde M(\infty)$. Using the Morse Lemma one can check that any diffeomorphism (and more generally any homotopy equivalence) $f\colon M\to M$ of a negatively curved manifold $M$ induces a homeomorphism $f_\infty\colon \tilde M(\infty)\to\tilde M(\infty)$ that only depends on 
$f_\#\colon \pi_1(M)\to\pi_1(M)$ (see, \eg~\cite{Kn}). It follows that $\tilde M(\infty)\to\tilde E(\infty)\stackrel{\tilde p_\infty}{\to} X$ is a fiber bundle whose structure group is $Top(\tilde M(\infty))$. The action of $\pi_1(M)$ on $\tilde E$ extends to $\tilde E(\infty)$ in the natural way.

Denote by $\tilde q\colon\tilde E\to\tilde M$ a lifting of fiber homotopy trivialization $q\colon E\to M$. Again, by the Morse Lemma, $\tilde q$ induces a continuous map $\tilde q_\infty\colon \tilde E(\infty)\to\tilde M(\infty)$ such that the restriction $\tilde q_\infty|_{\tilde M_x(\infty)}\colon\tilde M_x(\infty)\to\tilde M(\infty)$ is a homeomorphism for each $x\in X$.

\begin{remark}
 \label{remark_hom_triv}
 Note that $S\tilde M$ is (equivariently) homeomorphic to $\tilde M\times\tilde M(\infty)$ via $\tilde M\times\tilde M(\infty)\ni (y,\xi)\mapsto r_y^{-1}(\xi)\in S\tilde M$. Therefore the map
 $S\tilde E\ni v\mapsto (\tilde q(y),\tilde q_\infty(r_y(v))\in \tilde M\times\tilde M(\infty)$ induces a fiber homotopy trivialization $SE\to SM$ for the associated sphere bundle. 
\end{remark}

Now we use an idea of Gromov~\cite{Gr}  to finish the proof. We code each vector in $S\tilde E$ by a pair $(\gamma, y)$, where $y$ is the base point of the vector and $\gamma$ is the oriented geodesic in $\tilde M_y$ in the direction of the vector. Furthermore, we code each oriented geodesic $\gamma$ in $\tilde M_y$ by an ordered pair of distinct points $(\gamma(-\infty),\gamma(\infty))\in \tilde M_y(\infty)\times \tilde M_y(\infty)$. Define
$\tilde r_0\colon S\tilde E\to S\tilde M$ by
$$
\tilde r_0(\gamma(-\infty),\gamma(\infty),y)=(\tilde q_\infty(\gamma(-\infty)),\tilde q_\infty(\gamma(\infty)),pr(\tilde q(y))),
$$
where $pr\colon\tilde M\to \tilde q_\infty(\gamma)$ is the orthogonal projection on the geodesic $\tilde q_\infty(\gamma)$ with endpoints $\tilde q_\infty(\gamma(-\infty))$ and $\tilde q_\infty(\gamma(\infty))$.
Because the restrictions $\tilde q_\infty|_{\tilde M_x(\infty)}$ are homeomorphisms, the restrictions  $\tilde r_0|_{S\tilde M_x}$ induce homeomorphisms on the corresponding spaces of geodesics. However $\tilde r_0|_{S\tilde M_x}$ may fail to be injective along the geodesics. To overcome this difficulty we follow Gromov and consider the maps $\tilde r_K\colon S\tilde E\to S\tilde M$ given by
$$
\tilde r_K(\gamma(-\infty),\gamma(\infty),y)=\left(\tilde q_\infty(\gamma(-\infty)), \tilde q_\infty(\gamma(\infty)),
\frac1K\int_0^Kpr(\tilde q(\gamma(t)))dt\right),
$$
for $K>0$. For each $x\in X$ there exists a sufficiently large $K$ such that the restriction $\tilde r_K|_{S\tilde M_x}$ is an equivariant homeomorphism (see~\cite{Kn} for a detailed proof). In fact, because $q$ is continuous, the same argument yields a stronger result: for each $x\in X$ there exists a sufficiently large $K$ such that the restrictions $\tilde r_K|_{S\tilde M_y}$ are equivariant homeomorphisms for all $y$ in a neighborhood of $x$.
Hence, because $X$ is compact, for a sufficiently large $K$ the restrictions of $\tilde r$ to the fibers are homeomorphisms. Project $\tilde r$ back to $E$ to obtain the posited topological trivialization $r\colon E\to M$.
\end{proof}

\subsection{Further geometric questions}
\label{sec_12}
Given a closed smooth Riemannian manifold $M$ denote by $St_kM$ the bundle of ordered orthonormal $k$-frames in the tangent bundle $TM$. As before, given a negatively curved bundle $M\to E\stackrel{p}{\to} X$ we can define its {\it (fiberwise) associated $k$-frame bundle} $St_kM\to St_kE\stackrel{St_k(p)}{\to} X$.

\begin{theorem}\label{thm_2frames}
 Let $X$ be a closed manifold or a finite simplicial complex and $p\colon E\to X$ be a smooth negatively curved, fiber homotopically trivial bundle. Then its (fiberwise) associated $2$-frame bundle $St_2(p)\colon St_2E{\to} X$ is topologically trivial.
\end{theorem}
The proof is the same except that one has to use the idea of Cheeger~\cite{Gr} who suggested to code 2-frames in $St_2\tilde M$ by ordered triples of points in $\tilde M(\infty)$.

\begin{q}
 Assume that $M$ and $N$ are closed negatively curved manifolds with isomorphic fundamental groups. Is it true that $St_kM$ is homeomorphic to $St_kN$ for $k\ge 3$? 
\end{q}
\begin{q}
 Let $X$ be a finite simplicial complex and $p\colon E\to X$ be a smooth negatively curved, fiber homotopically trivial bundle. Let $k\ge 3$. Is it true that its (fiberwise) associated $k$-frame bundle $St_k(p)\colon St_kE{\to} X$ is topologically trivial?
\end{q}
Potentially, our techniques could be useful for this question, cf. discussion in Section~\ref{sec_64}.

\section{Anosov bundles: an overview}

The rest of the paper is mainly devoted to topological rigidity of smooth bundles equipped with fiberwise Anosov dynamics in various different setups. Some of these results directly generalize Theorems~\ref{thm_main} and~\ref{thm_main_general} from the preceding section. We informally summarize our topological rigidity results in the following table.
\\

\newcolumntype{"}{@{\hskip\tabcolsep\vrule width 2pt\hskip\tabcolsep}}
\makeatother

\makeatletter
\def\hlinewd#1{%
  \noalign{\ifnum0=`}\fi\hrule \@height #1 \futurelet
   \reserved@a\@xhline}
\makeatother

\newcommand\VRule[1][\arrayrulewidth]{\vrule width #1}

\hskip-0.8cm
\begin{tabular*}{13.67cm}{"c"c"c"}
\hlinewd{2pt} 
\backslashbox{Fiberwise \\Dynamics }{Topology} &
\shortstack{\\ Simply connected \\ base} & \shortstack{\\ Fiberwise homotopically \\ trivial bundle} \\
\hlinewd{2pt} 
Anosov diffeomorphisms & 
\begin{tabular}{c}
\\ 
Theorem~\ref{thm_anosov_sc}  \\
\\
\\
\hline 
examples:~\ref{thm_example}
\end{tabular}
& 
\begin{tabular}{c}
\\ 
Theorem~\ref{thm_anosov_nil}$^{(\star)}$, \\
strong form: Addendum~\ref{add_23}$^{(\star)}$  \\
\\
\hline 
examples:~\ref{thm_example},~\ref{add_anosov}
\end{tabular}
\\
\hlinewd{2pt} 
Anosov flows
& 
\begin{tabular}{c}
\\ 
Theorem~\ref{thm_anosov_flows_sc}  \\
\\
\\
\hline 
examples:~\ref{thm_2}
\end{tabular}
& 
\begin{tabular}{c}
\\ 
Theorem~\ref{thm_anosov_flows}$^{(\star\star)}$, \\
strong form: ?  \\
\\
\hline 
examples:~\ref{thm_2},~\ref{thm_example_flow}
\end{tabular}
\\
\hlinewd{2pt} 
\shortstack {\\ Partially hyperbolic \\
dynamics} &  
\begin{tabular}{c}
\\ 
? \\
\\
\end{tabular}
& 
\begin{tabular}{c} 
 \\
 ?  \\
\\
\hline 
discussion: sections~\ref{sec_12} and~\ref{sec_64}.
\end{tabular}
\\
\hlinewd{2pt} 
\end{tabular*}
\\

We now proceed with some comments on the table and contents of our paper. Then, after presenting the background on structural stability, we outline a general strategy which is common to all the proofs of topological rigidity results which appear further in the paper.

\subsection{Comments}
\begin{enumerate}
 \item The cells of the table contain references to topological rigidity results. The first row  of the table indicates the assumption on the fiber bundle under which corresponding result holds and the first column of the table lists the corresponding dynamical assumption.
 \item Under the label ``examples'' we refer to statements which provide examples to which corresponding topological rigidity result applies. All these examples are fairly explicit and are not smoothly rigid.
 \item The topological rigidity results from the last column require extra assumptions: ($\star$) infranilmanifold fibers; ($\star\star$) absence of homotopic periodic orbits. We do not know if these extra assumptions can be dropped.
 \item Theorems~\ref{thm_2} and~\ref{thm_212} (which do not appear in the table) show that we cannot dispose of the dynamical assumption and still have topological rigidity.
 \item In Section~\ref{sec_61} we point out that examples given by Theorem~\ref{thm_example} can be viewed as partially hyperbolic diffeomorphisms with certain peculiar properties.
\end{enumerate}

\subsection{Structural stability} All our proofs of topological rigidity results which appear in further sections strongly rely on structural stability of Anosov dynamical systems~\cite{An}.

{\bf Structural stability for Anosov diffeomorphisms. } {\it Let $f\colon M\to M$ be an Anosov diffeomorphism. Then, for every $\delta>0$ there exists $\varepsilon>0$ such that if a diffeomorphism $\bar f\colon M\to M$ is such that
$$
d_{C^1}(f,\bar f)<\varepsilon
$$
then $\bar f$ is also Anosov and there exist a homeomorphism $h\colon M\to M$ such that
$$
h\circ\bar f=f\circ h
$$
and
$$
d_{C^0}(h,id_M)<\delta.
$$
Moreover, such homeomorphism $h$ is unique.
}

{\bf Structural stability for Anosov flows. } {\it Let $g^t\colon M\to M$ be an Anosov flow. Then, for every $\delta>0$ there exists $\varepsilon>0$ such that if a flow $\bar g^t\colon M\to M$ is $C^1$ $\varepsilon$-close to $g^t$ 
then $\bar g^t$ is also Anosov and there exist an orbit equivalence $h\colon M\to M$, (\ie $h$ is a homeomorphism which maps orbits of $\bar g^t$ to orbits of $g^t$) 
such that
$$
d_{C^0}(h,id_M)<\delta.
$$
}
The homeomorphism $h$ given by structural stability is merely H\"older continuous and, in general, not even $C^1$. Note that the uniqueness part of structural stability fails for flows. Indeed, if $h$ is an orbit equivalence which is $C^0$ close to $id_M$ then $g^\beta\circ h \circ \bar g^\alpha$ is another orbit equivalence which is $C^0$ close to identity provided that the functions $\alpha, \beta\colon M\to\mathbb R$ are $C^0$ close to the zero function.

\subsection{Scheme of the proof of topological rigidity}
Below we outline our approach to topological rigidity of Anosov bundles in a broad-brush manner. Definitions and precise statements of results appear in further sections.

Assume that $M\to E\to X$ is a smooth fiber bundle over a compact base $X$ whose total space $E$ is equipped with a fiberwise Anosov diffeomorphism or a fiberwise Anosov flow. To show that such Anosov bundle is topologically trivial we first enlarge the structure group of the bundle
$$
\textup{Diff}(M)\subset \textup{Top}(M)
$$
and, from now on, view $M\to E\to X$ as a topological bundle. Employing Anosov dynamics in the fibers, we can identify nearby fibers using homeomorphisms which come from structural stability. This allows to reduce the structure group of the bundle
$$
\textup{Diff}(M)\subset \textup{Top}(M)\supset F,
$$
where $F$ is the group of certain self conjugacies of a fixed Anosov diffeomorphism of $M$ (self orbit equivalences in the flow case), which are homotopic to $id_M$.

The rest of the proof strongly depends on particular context. We only remark that, in the case when $E$ is equipped with a fiberwise Anosov diffeomorphism, group $F$ is discrete because of uniqueness part of structural stability. (Indeed, the only self conjugacy of Anosov diffeomorphism which is close to identity is the identity self conjugacy.) In the case of flows, the group $F$ is larger, however we are able to show that $F$ is contractible, \ie we further reduce the structure group to a trivial group which implies that the bundle is trivial as $\textup{Top}(M)$-bundle.


\section{Bundles admitting fiberwise Anosov diffeomorphisms}
\label{section_anosov_bundles}

Let $M\to E\to X$ be a smooth bundle. A fiber preserving diffeomorphism $f\colon E\to E$ is called a {\it fiberwise Anosov diffeomorphism} if the restrictions $f_x\colon M_x\to M_x$, $x\in X$, are Anosov diffeomorphisms. A smooth bundle $M\to E\to X$ is called {\it Anosov} if it admits a fiberwise Anosov diffeomorphism.
\begin{theorem}\label{thm_anosov_sc}
 Let $X$ be a closed simply connected manifold or a finite simply connected simplicial complex. Assume $M\to E\stackrel{p}{\to}X$ is an Anosov bundle. Then the bundle $p\colon E\to X$ is, in fact, topologically trivial.
\end{theorem}
\begin{theorem}\label{thm_anosov_nil}
 Let $X$ be a closed manifold or a finite simplicial complex. Assume that $M\to E\stackrel{p}{\to}X$ is a fiber homotopically trivial Anosov bundle, whose fiber $M$ is an infranilmanifold. Then the bundle $p\colon E\to X$ is, in fact, topologically trivial.
\end{theorem}
\begin{remark}
In the above theorem the assumption that the bundle is fiber homotopically trivial is a necessary one. To see this let $A, B\colon \T^n\to\T^n$ be two commuting hyperbolic automorphisms. Then the bundle  $\T^n\to E_A\stackrel{p}{\to}\S^1$, whose total space is defined as
$$
E_A=\T^n\times [0,1]/(x,0)\sim(Ax, 1),
$$
is Anosov. Indeed, the map
$$
E_A\ni (x,t)\mapsto (Bx, t)\in E_A
$$
is a fiberwise Anosov diffeomorphism.
\end{remark}
\begin{remark}
Theorem~\ref{thm_anosov_nil} can be generalized to a wider class of bundles whose fiber $M$ is only homeomorphic to an infranilmanifold. The modifications required in the proof are straightforward.
\end{remark}

We will explain in Section~\ref{sec_non_Anosov} how results of~\cite{FG2} imply that there are ``lots'' of smooth, fiber  homotopically trivial bundles $M\to E\to \mathbb S^d$ whose infranilmanifold fiber $M$ supports an Anosov diffeomorphism, but which are not Anosov bundles (and are not topologically trivial).

Two fiber bundles $M_1\to E_1\stackrel{p}{\to}X$ and $M_2\to E_2\stackrel{p_2}{\to}X$ over the same base $X$ are called {\it fiber homotopically equivalent} if there exists a fiber preserving continuous map $q\colon E_1\to E_2$ covering $id_X$ such that the restriction $q|_{M_{1,x}}\colon M_{1, x}\to M_{2,x}$ is a homotopy equivalence for each $x\in X$. Similarly, two fiber bundles $M_1\to E_1\stackrel{p}{\to}X$ and $M_2\to E_2\stackrel{p_2}{\to}X$ over the same base $X$ are called {\it topologically equivalent} if there exists a fiber preserving homeomorphism $h\colon E_1\to E_2$ covering $id_X$ such that the restriction $h|_{M_{1,x}}\colon M_{1,x}\to M_{2,x}$ is a homeomorphism for each $x\in X$. Theorem~\ref{thm_anosov_nil} can be generalized to the setting of non-trivial fiber bundles in the following way.
\begin{add}\label{add_23}
Let $X$ be a closed manifold or a finite simplicial complex. Assume that $M_1\to E_1\stackrel{p}{\to}X$ and $M_2\to E_2\stackrel{p_2}{\to}X$ are fiber homotopically equivalent Anosov fiber bundles over the same base $X$ whose fibers $M_1$ and $M_2$ are nilmanifolds. Then the bundles $p_1\colon E_1\to X$ and $p_2\colon E_2\to X$ are, in fact, topologically equivalent.
\end{add} 
\begin{remark}
Addendum~\ref{add_23} can be generalized to a wider class of bundles whose fiber $M_1$ and $M_2$ are only homeomorphic to nilmanifolds. The modifications required in the proof are  straightforward.
\end{remark}

We prove Theorem~\ref{thm_anosov_sc} now and postpone the proofs of Theorem~\ref{thm_anosov_nil} and Addendum~\ref{add_23} to Section~\ref{sec_thm_anosov}.
\begin{proof}[Proof of Theorem~\ref{thm_anosov_sc}]
Fix $x_0\in X$. For each $x\in X$ consider a path $\gamma\colon[0,1]\to X$ that connects $x_0$ to $x$. Because Anosov diffeomorphisms are structurally stable, for each $t\in[0,1]$ there exists a neighborhood of $t$, $\mathcal U\subset[0,1]$ and a continuous family of homeomorphisms $h_{s,t}\colon M_{\gamma(t)}\to M_{\gamma(s)}$ for $s\in \mathcal U$ such that
$$
h_{t,t}=id_{M_{\gamma(t)}}\;\;\;\mbox{and}\;\;\; h_{s,t}\circ f_{\gamma(s)}=f_{\gamma(t)}\circ h_{s,t}.
$$
Since $[0,1]$ is compact we can find a finite sequence $0=t_0<t_1<\ldots <t_N=1$ and homeomorphisms $h_{t_{i-1},t_i}$ that conjugate $f_{\gamma(t_{i-1})}$ to $f_{\gamma(t_i)}$. Clearly
$$
h_x=h_{t_{N-1},t_N}\circ h_{t_{N-2},t_{N-1}}\circ\ldots \circ h_{t_{0},t_1}
$$
is a conjugacy between $f_{x_0}$ and $f_x$. Because the ``local" conjugacies $h_{s,t}$ are unique and depend continuously on $s$, the resulting homeomorphism $h_x\colon M_x\to M_{x_0}$ does not depend on the arbitrary choices that we made. Moreover, two homotopic paths that connect $x_0$ to $x$ will yield the same homeomorphism $h_x$. Hence, since $X$ is simply connected, $h_x$ depends only on $x\in X$. It is also clear from the construction that $h_x$ depends on $x$ continuously. Therefore $E\ni y\mapsto h_{p(y)}(y)\in M_{x_0}$ gives the posited topological trivialization. 
\end{proof}

We will say that a smooth bundle $M\to E\stackrel{p}{\to}X$ is {\it smoothly trivial} if there exists a continuous map $r\colon E\to M$ such that the restriction $r|_{M_x}\colon M_x\to M$ is a diffeomorphism for each $x\in X$. If the base $X$ is a smooth closed manifold then the total space is also a smooth manifold. (This can be seen by approximating fiberwise smooth transition functions between the charts of the bundle by smooth transition functions.) In this case, the trivialization $r$ can be smoothed out so that the restrictions $r|_{M_x}\colon M_x\to M$ depend smoothly on $x\in X$. Hence, then map  $E\to X\times M$ given by $y\mapsto (p(y), r_{p(y)}(y))$ is fiber preserving diffeomorphism.
\begin{theorem}
 \label{thm_example}
 For each $d\ge 2$ there exists a smooth Anosov bundle $\T^n\to E\stackrel{p}{\to}\S^d$ which is not smoothly trivial.
\end{theorem}
\begin{remark}
 The construction of examples of Anosov bundles for the above theorem relies on the idea of construction of ``exotic'' Anosov diffeomorphism from~\cite{FG}.
\end{remark}
\begin{add}
 \label{add_anosov}
 There exists smooth fiber homotopically trivial Anosov bundle $\T^n\to E\stackrel{p}{\to}\S^1$ which is not smoothly trivial.
\end{add}
\begin{remark}
 Note that Anosov bundles posited above are topologically trivial by Theorems~\ref{thm_anosov_sc} and~\ref{thm_anosov_nil}.
\end{remark}

Because the proof of Addendum~\ref{add_anosov} is virtually the same as the proof of Theorem~\ref{thm_example} but requires some alternations in notation, we prove Theorem~\ref{thm_example} only.
To explain the construction of our examples of Anosov bundles we first need to present some background on the Kervaire-Milnor group of homotopy spheres.

\subsection{Gromoll filtration of Kervaire-Milnor group}

A {\it homotopy $m$-sphere} $\Sigma$ is a smooth manifold which is homeomorphic to the standard $m$-sphere $\S^m$. Kervaire and Milnor showed that if $m\ge 5$ then the set of oriented diffeomorphisms classes of homotopy $m$-spheres forms a finite abelian group $\Theta_m$ under the connected sum operation~\cite{KM}. 

One way to realize the elements of $\Theta_m$ is through so called twist sphere construction. A {\it twist sphere} $\Sigma_g$ is obtained from two copies of a closed disk $\mathbb D^{m}$ by pasting the boundaries together using an orientation preserving diffeomorphism $g\colon\S^{m-1}\to\S^{m-1}$. We will denote by $\textup{Diff}(\cdot)$ the space of orientation preserving diffeomorphisms. It is easy to see that the map $ \textup{Diff}(\S^{m-1})\ni g\mapsto \Sigma_g\in\Theta_m$ factors through to a homomorphism
$$
\pi_0\textup{Diff}(\S^{m-1})\to\Theta_m,
$$
which was known to be a group isomorphisms for $m\ge 6$~\cite{C}.

Fix $l\in[0,m-1]$ and view the $(m-1)$-disk $\DD^{m-1}$ as the product $\DD^l\times\DD^{m-1-l}$. 
Let $\textup{Diff}_l(\DD^{m-1},\partial)$ be the group of diffeomorphisms of the $(m-1)$-disk that preserve first $l$-coordinates and are identity in a neighborhood of the boundary $\partial\DD^{m-1}$. Then we have
$$
\textup{Diff}_l(\DD^{m-1},\partial)\hookrightarrow\textup{Diff}(\DD^{m-1},\partial)\hookrightarrow\textup{Diff}(\S^{m-1}),
$$
where the last inclusion is induced by a fixed embedding $\DD^{m-1}\hookrightarrow\S^{m-1}$. These inclusions induce homomorphisms
of corresponding groups of connected components
$$
\pi_0\textup{Diff}_l(\DD^{m-1},\partial)\hookrightarrow\pi_0\textup{Diff}(\DD^{m-1},\partial)\hookrightarrow\pi_0\textup{Diff}(\S^{m-1})\simeq\Theta_m.
$$
The image of $\pi_0\textup{Diff}_l(\DD^{m-1},\partial)$ in $\Theta_m$ does not depend on the choice of the embedding $\DD^{m-1}\hookrightarrow\S^{m-1}$ and is called the {\it Gromoll subgroup} $\Gamma_{l+1}^m$. Gromoll subgroups form a filtration
$$
\Theta_m=\Gamma_1^m \supseteq \Gamma_2^m \supseteq \: \ldots \:
\supseteq \Gamma_m^m = 0.
$$
Cerf~\cite{C} proved that $\Gamma_1^m=\Gamma_2^m$ for $m\ge 6$. Higher Gromoll subgroups are also known to be non-trivial in many cases, see~\cite{ABK}; see also~\cite{FO1}. In particular, $\Gamma_{2u-2}^{4u-1}\neq 0$ for $u\ge 4$.

\subsection{Proof of Theorem~\ref{thm_example}: construction of the bundle}
\label{sec_bundle_cnstr}
Fix an Anosov automorphism $L\colon \T^n\to\T^n$ whose unstable distribution $E^u$ has dimension $k$. Let $q$ be a fixed point of $L$. Choose a small product structure neighborhood $U\simeq \DD^k\times\DD^{n-k}$ in the proximity of fixed point $q$ so that $L(U)$ and $U$ are disjoint. Also choose an embedded disk $\DD^{d-1}\hookrightarrow\S^{d-1}$.  Then we have the product embedding $i\colon\DD^{d-1}\times\DD^{k}\times\DD^{n-k}\hookrightarrow \S^{d-1}\times\T^n$.

Now pick a diffeomorphism $\alpha\in\textup{Diff}_{k+d-1}(\DD^{n+d-1},\partial)$, consider $i\circ\alpha\circ i^{-1}$ and extend it by identity to a diffeomorphism $\alpha_1\colon\S^{d-1}\times\T^n\to\S^{d-1}\times\T^n$. Now we employ the standard clutching construction, that is, we paste together the boundaries $\DD^d_-\times\T^n$ and $\DD^d_+\times\T^n$ using $\alpha_1\colon\partial\DD^d_-\times\T^n\to\partial\DD^d_+\times\T^n$. This way we obtain a smooth bundle $\T^n\to E_\alpha\stackrel{p}{\to}\S^d$. Clearly this bundle does not depend on the choice of the product neighborhood $U$ and the choice of embedding $\DD^{d-1}\hookrightarrow\S^{d-1}$, \ie on the choice of embedding $i$. 
\begin{prop}\label{prop_anbundle1}
 If $[\alpha]\neq 0$ in the group $\Gamma_{k+d}^{n+d}$, $d\ge 2$, then $\T^n\to E_\alpha\stackrel{p}{\to}\S^d$ is not smoothly trivial.
\end{prop}
\begin{proof}
 Detailed proofs of similar results were given in~\cite{FJ89} and~\cite{FO1}. So we will only sketch the proof of this result; but enough to show how it follows from arguments in~\cite{FJ89} and~\cite{FO1}. 
 
 Our proof proceeds by assuming that there exists a fiber preserving diffeomorphism $f\colon E_\alpha\to \S^d\times\T^n$ and showing that this assumption leads to a contradiction. One may assume that $f|_{\DD^d_-\times\T^n}$ is the identity map (we view $\S^d\times\T^n$ as $\DD^d_-\times\T^n$ and $\DD^d_+\times\T^n$ glued together using the identity map). Also note that, by the Alexander trick, there is a ``canonical" homeomorphism $g\colon E_\alpha\to \S^d\times\T^n$ which is fiber preserving and satisfies $g|_{\DD^d_-\times\T^n}=id$. One also easily sees that $E_\alpha$ is diffeomorphic to $\S^d\times\T^n\#\Sigma_\alpha$ and that under this identification $g$ becomes the ``obvious homeomorphism"
 $$
 \S^d\times\T^n\#\Sigma_\alpha\to \S^d\times\T^n.
 $$
 Here, $\Sigma_\alpha$ is the exotic twist sphere of $\alpha$. 
 
 Now $f$ and $g$ are two smoothings of the topological manifold $\S^d\times\T^n$. Recall that a homeomorphism 
 $$
 \varphi\colon N\to \S^d\times\T^n,
$$
where $N$ is a closed smooth $(d+n)$-dimensional manifold, is called a {\it smoothing} of $\S^d\times\T^n$. Two smoothings $\varphi_i\colon N_i\to \S^d\times\T^n$, $i=0,1$, are {\it equivalent} if there exists a diffeomorphism $\psi\colon N_0\to N_1$ such that the composite $\varphi_1\circ\psi$ is {\it topologically pseudo-isotopic} (or {\it concordant}) to $\varphi_0$; \ie there exists a homeomorphism
$$
\Phi\colon N_0\times[0,1]\to\S^d\times\T^n\times[0,1]
$$
such that $\Phi|_{N_0\times\{i\}}\colon N_0\times\{i\}\to \S^d\times\T^n\times\{i\}$ is $\varphi_0$ if $i=0$ and $\varphi_1\circ\psi$ if $i=1$.

Since $\S^d\times\T^n$ is stably parallelizable, it is well-known that the smoothing $g\colon \S^d\times\T^n\#\Sigma_\alpha\to \S^d\times\T^n$ is inequivalent to the ``trivial smoothing" $id\colon \S^d\times\T^n\to \S^d\times\T^n$; cf.~\cite{FJ89} or~\cite{FO1}. But the smoothing $f\colon E_\alpha\to \S^d\times\T^n$ is obviously equivalent to $id\colon \S^d\times\T^n\to \S^d\times\T^n$, since $f$ is a diffeomorphism. Therefore, $f$ and $g$ are inequivalent smoothings of $\S^d\times\T^n$.

We now proceed to contradict this assertion. Since $f|_{\DD^d_-\times\T^n}=g|_{\DD^d_-\times\T^n}$, it suffices to show that $f|_{\DD^d_+\times\T^n}$ is topologically pseudo-isotopic  to $g|_{\DD^d_+\times\T^n}$ relative to boundary. But this follows from the basic Hsiang-Wall topological rigidity result~\cite{HW} that the homotopy-topological structure set $\mathcal S(\DD^k\times\T^n,\partial)$ contains exactly one element when $k+n\ge 5$ together with the fact (mentioned in the proof of Proposition~\ref{prop_trivialization}) that $\pi_iG(\T^n)=0$ for all $i>1$.
\end{proof}
In the proof above we only used the assumption that diffeomorphism $f\colon E_\alpha\to \S^d\times\T^n$ is fiber preserving in order to be able to assume that $f_{\DD^d_-\times \T^n}$ is identity. Notice that this step also works for diffeomorphisms which are fiber preserving only over the southern hemisphere $\DD^d_-$ of $\S^d$. Hence, from the above proof, we conclude that there does not exist  a diffeomorphism $E_\alpha\to \S^d\times\T^n$ which is fiber preserving over the southern hemisphere $\DD^d_-$ of $\S^d$. This observation justifies the following definition.
 A smooth bundle $\T^n\to E\to\S^d$ is {\it weakly smoothly trivial } if there exists a diffeomorphism $E\to \S^d\times\T^n$ which is fiber preserving over the southern hemisphere $\DD^d_-$ of $\S^d$. By the above observation, we have the following somewhat stronger result.
 \begin{add}\label{add_weak_trivial}
  If $[\alpha]\neq 0$ in the group $\Gamma_{k+d}^{n+d}$, $d\ge 2$, then $\T^n\to E_\alpha\stackrel{p}{\to}\S^d$ is not even weakly smoothly trivial.
 \end{add}
This result will be used in Section~\ref{sec_flows} while proving Theorem~\ref{thm_example_flow}.

\subsection{Proof of Theorem~\ref{thm_example}: construction of the fiberwise Anosov diffeomorphism}
\label{sec_const}
For each $(x,y)\in \DD^d_-\times \T^n$ define 
$$
f(x,y)=(x, L(y)).
$$
Because of clutching we automatically have 
\begin{equation}\label{eq_boundary}
 f(x,y)=\alpha_1\circ \bar L\circ\alpha_1^{-1} (x,y),
\end{equation}
where $(x,y)\in\partial\DD^d_+\times\T^n$ and $\bar L\colon\S^{d-1}\times\T^n\to\S^{d-1}\times\T^n$ is given by $(x,y)\mapsto (x, L(y))$. Hence it remains to extend $f$ to the interior of the disk $\DD^d_+$.

Take a smooth path of embeddings $i_t\colon \DD^{d-1}\times\DD^{k}\times\DD^{n-k}\hookrightarrow \S^{d-1}\times\T^n$, $t\in[0,1]$, that satisfies the following properties.
\begin{enumerate}
 \item $i_t=\bar L\circ i$ for all $t\in[0,\delta]$, where $\delta>0$ is small;
 \item $i_1=i$;
 \item $p\circ i_t= p\circ i$ for all $t\in[0,1]$;
 \item the images $\textup{Im}(i_t)$, $t\in[0,1]$, belong to a small neighborhood of $\S^{d-1}\times\{q\}$, but do not intersect $\S^{d-1}\times\{q\}$;
 \item $i_t(x,\cdot,\cdot)\colon\DD^{k}\times\DD^{n-k}\to\{x\}\times\T^n$ is a local product structure chart for $L\colon\T^n\to\T^n$ for all $t\in[0,1]$ and $x\in\DD^{d-1}$.
\end{enumerate}
Denote by $\alpha_t\colon \S^{d-1}\times\T^n\to\S^{d-1}\times\T^n$ the composition $i_t\circ\alpha\circ i_t^{-1}$ extended by identity to the rest of $\S^{d-1}\times\T^n$.

Now we use polar coordinates $(r,\varphi)$ on $\DD^d_+$, $r\in[0,1]$, $\varphi\in\S^{d-1}$, to extend $f$ to the interior of $\DD^d_+$. Let
$$
f((r,\varphi),y)=\alpha_r\circ \bar L\circ \alpha_1^{-1}((r,\varphi),y).
$$
When $r=1$ this definition is clearly consistent with~(\ref{eq_boundary}). When $r\in[0,\delta]$ we have
\begin{multline*}
 f((r,\cdot),\cdot)=\alpha_r\circ\bar L\circ \alpha_1^{-1}=i_r\circ\alpha\circ i_r^{-1}\circ\bar L\circ i_0\circ\alpha\circ i_0^{-1}\circ(\bar L\circ i_1)\circ\alpha^{-1}\circ i_1^{-1}\\
 = i_0\circ\alpha\circ i_0^{-1}\circ i_0\circ\alpha^{-1}\circ i_1^{-1}=i_0\circ i_1^{-1}=\bar L,
\end{multline*}
for points in the image of $i_1$ and it is obvious that $f((r,\cdot),\cdot)=\bar L$ outside of the image of $i_1$. Note that because the path $\alpha_r$ is locally constant at $r=0$ diffeomorphism $f$ is smooth at $r=0$. 

Therefore we have constructed a fiber preserving diffeomorphism $f\colon E_\alpha\to E_\alpha$. The restrictions to the fibers $f_x\colon \T^n_x\to \T^n_x$ are clearly Anosov when $x\in \DD^{d}_-$ or close to the origin of $\DD^{d}_+$. Note that our construction depends on the choice of the diffeomorphism $\alpha\in\textup{Diff}_{k+d-1}(\DD^{n+d-1},\partial)$, the choice of embedding $i$ and the choice of isotopy $i_t$, $t\in[0,1]$.

\begin{prop}\label{prop_anbundle2}
 For each number $d\ge 2$, each Anosov automorphism $L\colon\T^n\to\T^n$ whose unstable distribution is $k$-dimensional and each diffeomorphism $\alpha\in\textup{Diff}_{k+d-1}(\DD^{n+d-1},\partial)$ there exists an embedding $i\colon\DD^{d-1}\times\DD^{k}\times\DD^{n-k}\hookrightarrow \S^{d-1}\times\T^n$ and an isotopy $i_t$, $t\in[0,1]$ such that the fiber preserving diffeomorphism $f=f(\alpha,i, i_t)\colon E_\alpha\to E_\alpha$ of the total space of the fiber bundle $\T^n\to E_\alpha \stackrel{p}{\to}\S^d$ is a fiberwise Anosov diffeomorphism.
\end{prop}

Clearly Propositions~\ref{prop_anbundle1} and~\ref{prop_anbundle2} together with non-triviality results for Gromoll subgroups imply Theorem~\ref{thm_example}.

\begin{proof}[Sketch of the proof of Proposition~\ref{prop_anbundle2}]
 Let $h$ be the round metric on $\S^{d-1}$ and $g$ be the flat metric on $\T^n$. Choose and fix an embedding $i\colon\DD^{d-1}\times\DD^{k}\times\DD^{n-k}\hookrightarrow \S^{d-1}\times\T^n$ and isotopy $i_t$, $t\in[0,1]$ that satisfy properties 1-5 listed in Subsection~\ref{sec_const}. Then for some small $\varepsilon>0$ and all $t\in[0,1]$
 $$
 \textup{Im}(i_t)\in V_\varepsilon^{h+g}(\S^{d-1}\times\{q\}), 
 $$
 where $V_\varepsilon^{h+g}(\S^{d-1}\times\{q\})$ is the $\varepsilon$-neighborhood of $\S^{d-1}\times\{q\}$ in $\S^{d-1}\times\T^n$ equipped with Riemannian metric $h+g$. (Recall that $q$ is a fixed point of $L$.)
 
 For each $m\ge1$ consider the Riemannian metric $h+mg$ and corresponding $\varepsilon$-neighborhood $V_\varepsilon^{h+mg}(\S^{d-1}\times\{q\})$. These neighborhoods are isometric in an obvious way 
 $$I\colon V_\varepsilon^{h+g}(\S^{d-1}\times\{q\})\to V_\varepsilon^{h+mg}(\S^{d-1}\times\{q\})$$
 Then there is a unique choice of embeddings $i_t^m$, $t\in[0,1]$, that makes the diagram
  
 $$
 \xymatrix
 {
  \DD^{d-1}\times\DD^{k}\times\DD^{n-k}\;\ar[r]^{i_t}\ar[rd]^{i^m_t} & V_\varepsilon^{h+g}(\S^{d-1}\times\{q\}) \ar[d]^I &\\
 &  V_\varepsilon^{h+mg}(\S^{d-1}\times\{q\})
 }
 $$
 commute. The following lemma clearly yields Proposition~\ref{prop_anbundle2}.
 \begin{lemma}
  For a sufficiently large number $m\ge1$ the diffeomorphism $f=f(\alpha, i^m, i^m_t)\colon E_\alpha\to E_\alpha$ constructed following the procedure of Subsection~\ref{sec_const} is a fiberwise Anosov diffeomorphism. 
 \end{lemma}
We do not provide a detailed proof of this lemma because a very similar argument was given in Section~3.3 of~\cite{FG}. 

The first step  is to notice that the unstable foliation $\W_L^u$ of $L\colon\T^n\to\T^n$ is invariant under the restrictions to the fiber $f_x\colon\T_x^n\to\T_x^n$, $x\in\DD^{d}_+$. This is because $i_t^m\circ\alpha\circ(i_t^m)^{-1}$ preserves foliation $\W_L^u$. Then one argues that for a sufficiently large $m$ the return time (under $f_x$) to $\cup_{t\in[0,1]}\textup{Im}(i_t)$ is large, which implies that $\W_L^u$ is indeed an expanding foliation for $f_x$.

The second step is to employ a standard cone argument to show that there exists an $Df_x$-invariant stable distribution transverse to $T\W_L^u$.
\end{proof}

\subsection{Non-Anosov fiber homotopically trivial bundles}
\label{sec_non_Anosov}
The following result demonstrates that the assumption {\it ``$M\to E\to X$ is an Anosov bundle''} in Theorem~\ref{thm_anosov_nil} cannot be replaced by  the assumption {\it ``the fiber $M$ supports an Anosov diffeomorphism.''}
\begin{theorem}[cf. Theorem~\ref{thm_2}]
\label{thm_212}
Let $M$ be any infranilmanifold that has dimension $\ge 10$, supports an affine Anosov diffeomorphism and has non-zero first Betti number. Then for each prime $p$ there exists a smooth bundle $M\to E\to \S^{2p-3}$ such that
 \begin{enumerate}
  \item the bundle $M\to E\to \S^{2p-3}$ is fiber homotopically trivial;
  \item the bundle $M\to E\to \S^{2p-3}$ is not topologically trivial (and, hence, is not Anosov by Theorem~\ref{thm_anosov_nil}).
 \end{enumerate}
 \end{theorem}
 
 For the proof of Theorem~\ref{thm_212} we need to recall the clutching construction which we use as a way to create fiber bundles over spheres.  Let $\S^d$, $d\ge 2$, be a sphere, let $N$ be a closed manifold and let $\alpha\colon \S^{d-1}\to\ \textup{Diff}(N)$ be a map. Let $\DD^d_-$ and $\DD^d_+$ be the southern and the northern hemispheres of $\S^d$, respectively. Consider manifolds $\DD^d_-\times N$ and $\DD^d_+\times N$ as trivial fiber bundles with fiber $N$. Note that there boundaries are both diffeomorphic to $\S^{d-1}\times N$. 
 
 Now define diffeomorphism $\bar\alpha\colon \S^{d-1}\times N\to \S^{d-1}\times N$ by the formula 
 $$(x,y)\mapsto (x,\alpha(x)(y)).$$
  {\it Clutching with $\alpha$} amounts to pasting together the boundaries of $\DD^d_-\times N$ and $\DD^d_+\times N$ using $\bar\alpha$. The resulting manifold $E$ is naturally a total space of smooth fiber bundle $N\to E\to\S^d$ because diffeomorphism $\bar\alpha$ maps fibers to fibers.

\begin{proof}
 We first consider the case when $p>2$. The assumptions (on the infranilmanifold $M$) of Proposition~5 of~\cite{FG2} are satisfied and we obtain a class $[\alpha]\in\pi_{2p-4}(\textup{Diff}_0(M))$ whose image under the natural inclusion $\pi_{2p-4}(\textup{Diff}_0(M))\to \pi_{2p-4}(\textup{Top}_0(M))$ is non-zero. 
 Then clutching with $\alpha\colon\S^{2p-4}\to \textup{Diff}_0(M)$ yields a smooth bundle $M\to E\to \S^{2p-3}$ which is not topologically trivial. 
 
 On the other hand, by~(\ref{eq_homotopy_homotopy}), the image of $\alpha$ in $\pi_{2p-4}(G_0(M))$ is zero. Hence  $M\to E\to \S^{2p-3}$ is fiber homotopically trivial. 
 
 The case when $p=2$ is analogous and uses Proposition~4 of~\cite{FG2}.
\end{proof}

\section{Bundles admitting fiberwise Anosov flows}
\label{sec_flows}
Here we generalize Theorems~\ref{thm_main} and~\ref{thm_main_general} to the setting of abstract Anosov flows.
\begin{theorem}
\label{thm_anosov_flows_sc}
Let $X$ be a  closed simply connected manifold or a finite simply connected simplicial complex and let $p\colon E\to X$ be
a fiber bundle whose fiber $M$ is a closed manifold. Assume that $E$ admits a $C^\infty$ flow which leaves the fiber $M_x$, $x\in X$, invariant and the restrictions of this flow to the fibers $g_x^t\colon M_x\to M_x$, $x\in X$, are transitive Anosov flows. Then the bundle $p\colon E\to X$ is topologically trivial.
\end{theorem}
\begin{theorem}
\label{thm_anosov_flows}
Let $X$ be a closed manifold or finite simplicial complex and $p\colon E\to X$ be a fiber homotopically trivial bundle whose fiber $M$ is a closed manifold. Assume that $E$ admits a $C^\infty$ flow which leaves the fibers $M_x$, $x\in X$, invariant and the restrictions of this flow to the fibers $g_x^t\colon M_x\to M_x$, $x\in X$, satisfy the following conditions 
\begin{enumerate}
\item  flows $g_x^t\colon M_x\to M_x$ are transitive Anosov flows;
\item flows $g_x^t\colon M_x\to M_x$ do not have freely homotopic periodic orbits.
\end{enumerate} 
Then the bundle $p\colon E\to X$ is topologically trivial.
\end{theorem}
\begin{remark}
  We only consider free homotopies of periodic orbits that preserve flow direction of the orbits. This way geodesic flows on negatively curved manifolds do not have freely homotopic periodic orbits.
\end{remark}
\begin{remark}
 By applying structural stability we see that the flows $g_x^t$ are all orbit equivalent. Hence if assumption 2 of Theorem~\ref{thm_anosov_flows} holds for one $x_0\in X$ then it automatically holds for all $x\in X$.
\end{remark}

Theorem~\ref{thm_anosov_flows_sc} obviously implies~Theorem~\ref{thm_main}.
Let us also explain how Theorem~\ref{thm_anosov_flows} implies Theorem~\ref{thm_main_general}. The associated sphere bundle $S(p)\colon SE\to X$ is fiber homotopically trivial by Remark~\ref{remark_hom_triv}. The total space $SE$ is equipped with fiberwise geodesic flow $g^t$. The flows on the fibers $g_x^t$, $x\in X$, are Anosov. Moreover they are transitive because they are ergodic with respect to fully supported Liouville measures by the work of Anosov~\cite{An}. Finally, assumption~2 of Theorem~\ref{thm_anosov_flows} is satisfied because each free homotopy class of loops on a negatively curved manifolds contains only one closed geodesic (which yields two periodic orbits which are not homotopic as oriented periodic orbits).

Another example of Anosov flow which satisfies assumption~2 of Theorem~\ref{thm_anosov_flows} is the suspension flow of an Anosov diffeomorphism of an infranilmanifold. 

\begin{theorem}\label{thm_example_flow} For any $d\ge 2$ there exists a smoothly non-trivial fiber bundle $M\to E'\stackrel{p'}{\to}\S^d$ which admits fiberwise Anosov flow which satisfies the assumptions of Theorem~\ref{thm_anosov_flows}.
\end{theorem}
This result is analogous to (and is based on) Theorem~\ref{thm_example}. 

For expository reasons we present the proof of Theorem~\ref{thm_anosov_flows} ahead of the proof of Theorem~\ref{thm_anosov_flows_sc}. Then we complete this section with the proof Theorem~\ref{thm_example_flow}.

\begin{proof}[Proof of Theorem~\ref{thm_anosov_flows}]
 Let $q\colon E\to M$ be a fiber homotopy trivialization. We identify $M$ with a particular fiber $M_{x_0}=p^{-1}(x_0)$. Then we can assume without loss of generality that $q|_{M}\colon M\to M$ is homotopic to identity. We will abbreviate the notation for the Anosov flow $g^t_{x_0}$ on $M_{x_0}=M$ to simply $g^t$. We use $Y$ to denote the Anosov vector field on $M$ defined by
 $$
 Y(a)=\frac{dg^t(a)}{dt}\Bigg{|}_{t=0},\;\; a\in M.
 $$
 Similarly $Y_x$ denotes the generating vector field for $g_x^t$, $x\in X$.
 
 For each $x\in X$ consider the space $F_x$ of all homeomorphisms $\varphi\colon M\to M_x$ which satisfy the following properties. 
 \renewcommand\labelenumi{P\theenumi.}
 \begin{enumerate}
 \item $\varphi\colon M\to M_x$ is an orbit conjugacy between $g^t$ and $g^t_x$, which preserves the flow direction;
 \item the map $q\circ \varphi$ is homotopic to $id_M$;
 \item $\varphi\colon M\to M_x$ is bi-H\"older continuous, \ie both $\varphi$ and $\varphi^{-1}$ are H\"older continuous with some positive H\"older exponent (which can depend on $\varphi$);
 \item $\varphi\colon M\to M_x$ is differentiable along $Y$; moreover, the derivative $\varphi'\colon M\to R$ along $Y$ defined by
 $$
 D\varphi(a)Y(a)=\varphi'(a)Y_x(\varphi(a))
 $$
 is a positive H\"older continuous function (whose H\"older exponent depends on $\varphi$). 
 \end{enumerate}
 \begin{remark}\label{remark_inverse}
 Note that if $\varphi\in F_x$ then $\varphi^{-1}\colon M_x\to M$ is differentiable along $Y_x$. Moreover the derivative of $\varphi^{-1}$ along $Y_x$ is given by $1/\varphi'\circ\varphi^{-1}$ and hence is also H\"older continuous.
 \end{remark}
In particular, we have defined the space $F\stackrel{\mathrm{def}}{=}F_{x_0}$ of self orbit conjugacies of $g^t$ that are homotopic to identity and satisfy regularity properties P3 and P4. It is routine to check that $F$ is a non-empty group when equipped with composition operation. By the same routine check $F$ acts on $F_x$ by pre-composition.
\begin{lemma}\label{lemma1}
The action of $F$ on the space $F_x$ is free and transitive for each $x\in X$.
\end{lemma}
\begin{proof}
Fix $x\in X$. It is obvious that $F$ acts freely on $F_x$.

Now we check that $F_x$ is non-empty. Consider a path $\gamma\colon[0,1]\to X$ that connects $x_0$ to $x$. Then we have a one parameter family of Anosov flows $g^t_{\gamma(s)}$, $s\in[0,1]$. By the structural stability, each $s\in[0,1]$ has a neighborhood $\mathcal U\subset[0,1]$ such that for each $r\in\mathcal U$ the flow $g^t_{\gamma(s)}$ is orbit conjugate to $g^t_{\gamma(r)}$ via orbit conjugacy $\xi_{s,r}$ which is $C^0$ close to identity (in a chart). This orbit conjugacy is bi-H\"older and it is easy to choose it so that it is differentiable along $Y_{\gamma(s)}$. The fact that $\xi_{s,r}$ can be chosen so that the derivative of $\xi_{s,r}$ along $Y_{\gamma(s)}$ is H\"older continuous is more subtle. One way to see this is from the proof of structural stability via the implicit function theorem, see \eg~\cite[Theorem A.1]{LMM},  where continuity of the derivative of $\xi_{s,r}$ along $Y_{\gamma(s)}$ is established. The same proof goes through to yield H\"older continuity, see~\cite[Proposition 2.2]{KKPW}. (Note that at this point we do not care about dependence of $\xi_{s,r}$ on the parameter $r$.) Next, since $[0,1]$ is compact, we can find a finite sequence $0=s_0<s_1<\ldots <s_N=1$ and orbit conjugacies $\xi_{s_{i-1},s_i}$ between  $g^t_{\gamma(s_{i-1})}$ and $g^t_{\gamma(s_{i})}$, $i=1\ldots N$. Then it is clear that
$$
\varphi_x\stackrel{\mathrm{def}}{=}\xi_{s_{N-1},s_N}\circ \xi_{s_{N-2},s_{N-1}}\circ\ldots\circ  \xi_{s_{1},s_0}
$$
belongs to the space $F_x$.

We proceed to checking that $F$ acts transitively on $F_x$. Take any $\psi\in F_x$. Then, since $\psi=\varphi_x\circ(\varphi_x^{-1}\circ\psi)$, we only need to check that $\varphi_x^{-1}\circ\psi\in F$. Clearly $\varphi_x^{-1}\circ\psi$ satisfies properties P1, P3 and P4. Property P2 also holds. Indeed,
$$
\varphi_x^{-1}\circ\psi\simeq \varphi_x^{-1}\circ id_{M_x}\circ\psi \simeq \varphi_x^{-1}\circ (q|_{M_x})^{-1}\circ q|_{M_x}\circ\psi\simeq id_M\circ id_M=id_M.
$$
(Here $(q|_{M_x})^{-1}$ stands for a homotopy inverse of $q|_{M_x}$).
\end{proof}
Let
$$
\mathcal F(X)=\bigsqcup_{x\in X}F_x.
$$
We equip $\mathcal F(X)$ with $C^0$ topology using charts of $p\colon E\to X$.
Group $F$ acts on $\mathcal F(X)$ by pre-composition. Clearly this action is continuous.

\begin{lemma}[Parametrized structural stability]\label{lemma2}
For each $x\in X$ there exists a neighborhood $\mathcal U$ of $x$ and a continuous family of homeomorphisms $\xi_z\colon M_x\to M_z$, $z\in\mathcal U$, such that
 \renewcommand\labelenumi{\theenumi.}
\begin{enumerate}
\item $\xi_x=id_{M_x}$;
\item $\xi_z$ is an orbit conjugacy between $g_x^t$ and $g_z^t$ for each  $z\in\mathcal U$;
\item $\xi_z$ is bi-H\"older continuous for each $z\in \mathcal U$;
\item $\xi_z$ is differentiable along $Y_x$ for all $z\in\mathcal U$; moreover, the derivative of $\xi_z$ along $Y_x$ is H\"older continuous and depends continuously on $z$, $z\in\mathcal U$.
\end{enumerate}
\end{lemma}
\begin{remark} In fact, dependence on $z\in\mathcal U$ is smooth, but we won't need it.
\end{remark}
By using a chart about $x$ for the bundle $p\colon E\to X$ one readily reduces Lemma~\ref{lemma2} to a lemma about a smooth multi-parameter family of Anosov flows on $M_x$. In this form Lemma~\ref{lemma2}  was proved in~\cite[Proposition 2.2]{KKPW} using Moser's implicit function theorem approach to structural stability. A more detailed proof is in~\cite[Theorem A.1]{LMM}, however~\cite{LMM} does not address H\"older continuity of the derivative along the flow. Both~\cite{KKPW} and~\cite{LMM} treat one parameter families of Anosov flows but the proof can be adapted to multi-parameter setting in a straightforward way.
\begin{lemma}\label{lemma3} The map $\mathcal F(p)\colon \mathcal F(X)\to X$ that assigns to each $\varphi\in F_x\subset \mathcal F(X)$ its base-point $x$ is a locally trivial principal fiber bundle with structure group $F$.
\end{lemma}
\begin{proof}
Take any $x\in X$. Then Lemma~\ref{lemma2} gives a neighborhood $\mathcal U$ and a continuous family of homeomorphisms $\xi_z\colon M_x\to M_z$, $z\in\mathcal U$. Consider a chart $\mathcal U\times F\to {\mathcal F(p)}^{-1}(\mathcal U)$ given by
$$
(z,\varphi)\mapsto \xi_z\circ\varphi_x\circ\varphi.
$$
By Lemma~\ref{lemma2} this chart is continuous. By Lemma~\ref{lemma1} it is injective and surjective. Also it is easy to see that the inverse is continuous. Finally, Lemma~\ref{lemma1} also implies that $\mathcal F(p)\colon\mathcal F(X)\to X$ is a principal bundle with respect to the fiberwise action of $F$ by pre-composition.
\end{proof}
\begin{lemma}\label{lemma4}
Group $F$ is contractible.
\end{lemma}
This lemma implies that the bundle $\mathcal F(p)\colon\mathcal F(X)\to X$ has a global (continuous) cross-section $\sigma\colon X\to \mathcal F(X)$ (and hence is a trivial bundle). Define
$
\tau\colon X\times M\to E
$
by
$$
\tau(x,a)=\sigma(x)(a).
$$
Clearly $\tau$ is a fiber preserving homeomorphism. The inverse $\tau^{-1}$ is the posited topological trivialization. Hence in order to finish the proof we only need to establish Lemma~\ref{lemma4}.\\
{\it Proof of Lemma~\ref{lemma4}.}
Let $\gamma$ be a periodic orbit of $g^t\colon M\to M$ and let $\varphi\in F$. Then $\varphi(\gamma)$ is also a periodic orbit which is homotopic to $\gamma$. Hence, by assumption 2 of the theorem, $\varphi(\gamma)=\gamma$ for each periodic orbit $\gamma$ and each $\varphi\in F$.

Pick a $\varphi\in F$ and recall that according to Remark~\ref{remark_inverse} $\varphi^{-1}$ is differentiable along $Y$ with a H\"older continuous derivative $(\varphi^{-1})'$. By applying Newton-Leibniz formula we obtain the following statement.
\begin{claim} For every periodic orbit $\gamma$ 
$$
\int_0^{T(\gamma)}(\varphi^{-1})'(\gamma(t))dt=T(\gamma),
$$
where $T(\gamma)$ is the period of $\gamma$.
\end{claim}
Therefore the integral of $(\varphi^{-1})'-1$ over each periodic orbit vanishes. Recall also that $g^t$ is a transitive Anosov flow. Hence we can apply the Livshitz Theorem (see, \eg~\cite[Theorem 19.2.1]{KH}) to $(\varphi^{-1})'-1$ and conclude that this function is a coboundary; that is, there exists a H\"older continuous function $\alpha\colon M\to\R$ which is differentiable along $Y$ and
\begin{equation}\label{star}
\alpha'=(\varphi^{-1})'-1,
\end{equation}
where $\alpha'$ is the derivative of $\alpha$ along $Y$.

Define the following homotopy along the orbits of $g^t$
\begin{equation}\label{2star}
\varphi_s(x)=g^{s\alpha(x)}(\varphi(x)),\;\; s\in[0,1]
\end{equation}
Clearly $\varphi_0=\varphi$. Now we calculate the derivative of $\varphi_1$ along $Y$
$$
\varphi_1'(x)=(g^{\alpha(x)})'(\varphi(x))\varphi'(x)=(1+\alpha'(\varphi(x)))\varphi'(x)\stackrel{(\ref{star})}{=}(\varphi^{-1})'(\varphi(x))\varphi'(x)=1.
$$
Thus $\varphi_1$ is time preserving self orbit conjugacy of $g^t$, \ie $\varphi_1$ is a true self conjugacy. Also it clear from~(\ref{2star}) that $\varphi_1$ and $\varphi$ coincide on the space of orbits and hence, $\varphi_1(\gamma)=\gamma$ for every periodic orbit $\gamma$. Hence we are able to apply the main result of~\cite{JH} to conclude that $\varphi_1=g^{t_0}$ for some $t_0\in\R$. 

When $s=1$ equation~(\ref{2star}) becomes
$$
g^{t_0}(x)=g^{\alpha(x)}(\varphi(x))
$$ 
or
$$
\varphi(x)=g^{t_0-\alpha(x)}(x).
$$
Hence for each $\varphi\in F$ there exists a continuous function $\beta_\varphi\colon M\to \R$ such that $\varphi=g^{\beta_\varphi}$.
\begin{claim}\label{claim2}
The map $F\ni\varphi\mapsto\beta_\varphi\in C^0(M,\R)$ is continuous.
\end{claim}
\begin{proof}
Let $\rho>0$ be the constant associated to the local product structure of the stable and unstable foliations of $g^t$. 

Assume that $\varphi$ is a discontinuity point of the map $\varphi\mapsto\beta_\varphi$. Then there exists $\varepsilon>0$ and $\psi\in F$ such that 
\begin{equation}\label{3star}
d_{C^0}(\varphi,\psi)<\varepsilon\;\;\;\; \text{and}  \;\;\;d_{C^0}(\beta_\varphi,\beta_\psi)>2\varepsilon.
\end{equation}
We can assume that $\varepsilon\lll\rho$.

Pick a point $x\in M$ such that $|\beta_\varphi(x)-\beta_\psi(x)|>2\varepsilon$. 
Consider a neighborhood $\mathcal U$ of size $\rho$ about $\varphi(x)$.
Let $\mathcal O(x)=\{g^t(x): t\in\R\}$. Then $\psi(x)\in\mathcal O(x)\cap\mathcal U$. By~(\ref{3star}) points $\psi(x)$ and $\varphi(x)$ belong to different connected components of $\mathcal O(x)\cap\mathcal U$ as indicated on Figure~\ref{fig1}. It follows that, in fact, $|\beta_\varphi(x)-\beta_\psi(x)|>\rho$.

\begin{figure}[htbp]
\begin{center}
\includegraphics{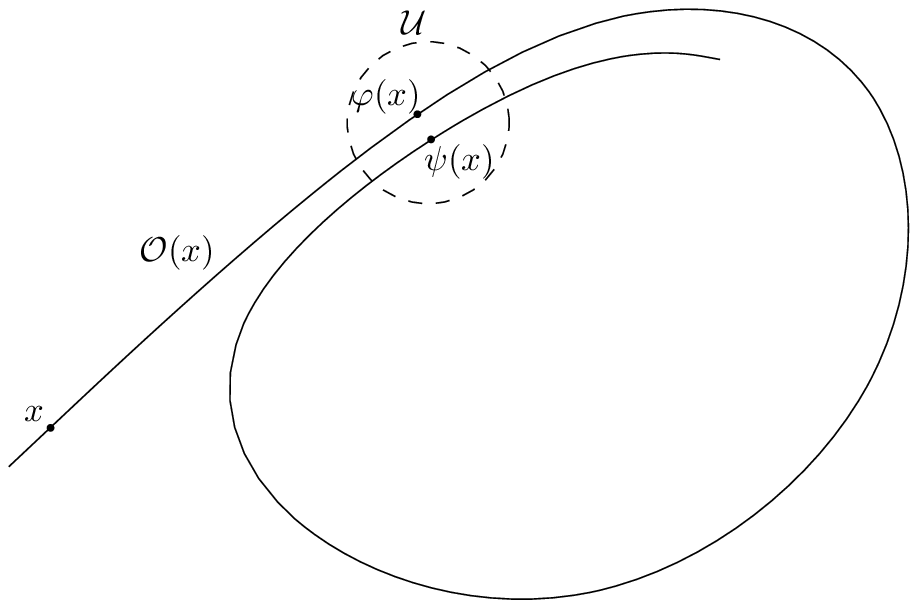}
\end{center}
 \caption{}
\label{fig1}
\end{figure}

By continuity the picture on Figure~\ref{fig1} is robust. Therefore we can perturb $x$ (we continue to denote this perturbation by $x$) to ensure that $\mathcal O^+(x)=\{g^t(x): t\ge 0\}$ is dense in $M$. 

Note that as a point $y$ moves along $\mathcal O^+(x)$ the points $\varphi(y)$ and $\psi(y)$ sweep the positive semi-orbits $\mathcal O^+(\varphi(x))$ and $\mathcal O^+(\psi(x))$ respectively. But $d(\varphi(y),\psi(y))<\varepsilon$ for all $y$. Hence the picture in the local product structure neighborhood of $\varphi(y)$ remains the same as $y$ moves along $\mathcal O^+(x)$. We conclude that $\psi(x)$ belongs to the weak stable manifold of $\varphi(x)$. Hence the distance from $\psi(y)$ to the local orbit $\{g^t(\varphi(y)): t\in(-\varepsilon, \varepsilon)\}$ goes to zero as $y$ moves along $\mathcal O^+(x)$. Because $\mathcal O^+(x)$ is dense, we have that, in fact, the distance between $\psi(z)$ and the local orbit $\{g^t(\varphi(z)): t\in(-\varepsilon, \varepsilon)\}$ is zero for all $z\in M$. This means that $d_{C^0}(\beta_\varphi,\beta_\psi)<\varepsilon$, which yields a contradiction.
\end{proof}
Define $\Delta\colon[0,1]\times F\to F$ by
$$
\Delta_s(\varphi)=g^{(1-s)\beta_\varphi}.
$$
Claim~\ref{claim2} implies that $\Delta$ is continuous. Clearly $\Delta_0=id_F$ and $\Delta_1$ maps $F$ to $id_M$. Hence $\Delta$ deformation retracts $F$ to a point. This finishes the proof of Lemma~\ref{lemma4} and, hence, the proof of Theorem~\ref{thm_anosov_flows}.
\end{proof}

\begin{proof}[Proof of Theorem~\ref{thm_anosov_flows_sc}]
A substantial part of the proof is exactly the same as in the preceding proof, yet there some very different ingredients which, in particular, allow to avoid the use of the assumption on non-homotopic periodic orbits of Theorem~\ref{thm_anosov_flows} and utilize triviality of $\pi_1(X)$ instead. In order to keep the exposition coherent we have allowed a few repetitions from the preceding proof.

As before, we fix a base-point $x_0\in X$ and identify $M$ with $M_{x_0}=p^{-1}(x_0)$. Further we abbreviate the notation for the Anosov flow $g_{x_0}^t$ on $M_{x_0}=M$ to $g^t$. We use $Y$ to denote the Anosov vector field on $M$ and $Y_x$ for the vector field which generates $g_x^t$, $x\in X$.

For each $x\in X$ consider the space $F_x$ of all homeomorphisms $\varphi\colon M\to M_x$ which satisfy the following properties. (Cf. the definition of $F_x$ in the proof of Theorem~\ref{thm_anosov_flows}.)
 \renewcommand\labelenumi{P\theenumi.}
 \begin{enumerate}
 \item $\varphi\colon M\to M_x$ is an orbit conjugacy between $g^t$ and $g^t_x$, which preserves the flow direction;
 \item[P3.] $\varphi\colon M\to M_x$ is bi-H\"older continuous, \ie both $\varphi$ and $\varphi^{-1}$ are H\"older continuous with some positive H\"older exponent (which can depend on $\varphi$);
 \item[P4.] $\varphi\colon M\to M_x$ is differentiable along $Y$; moreover, the derivative $\varphi'\colon M\to R$ along $Y$ defined by
  $$
 D\varphi(a)Y(a)=\varphi'(a)Y_x(\varphi(a))
 $$
 is a positive H\"older continuous function (whose H\"older exponent depends on $\varphi$);
 \end{enumerate}
 
Note that, in particular, we have defined the space $F\stackrel{\mathrm{def}}{=}F_{x_0}$ of certain special self orbit conjugacies of $g^t$. Following closely the arguments in the proof of Theorem~\ref{thm_anosov_flows} we obtain the following claims.
\begin{enumerate}
 \item[ 1.] Space $F$ is a non-empty group when equipped with composition operation.
\item[ 2.] The action of $F$ on the space $F_x$ by pre-composition is free and transitive for each $x\in X$.
\end{enumerate}
Now let
$$
\mathcal F(X)=\bigsqcup_{x\in X}F_x.
$$
We equip $\mathcal F(X)$ with $C^0$ topology using charts of $p\colon E\to X$. 

\medskip
\begin{enumerate}
\item[\  3.] The map $\mathcal F(p)\colon \mathcal F(X)\to X$ that assigns to each $\varphi\in F_x\subset \mathcal F(X)$ its base-point $x$ is a locally trivial principal fiber bundle with structure group $F$.
\end{enumerate}
\medskip

We call the bijection between the space of orbits of $g^t$ and the space of orbits of $g^t_x$ induced by a $\varphi\in F_x$ a {\it marking} of orbits of $g^t_x$. There is an obvious map $F_x\to \EuScript M_x$ to the space $\EuScript M_x$ of all possible markings of orbits of $g^t_x$. Let
$$
\EuScript M(X)=\bigsqcup_{x\in X}\EuScript M_x.
$$
The assembled map $mark\colon \mathcal F(X)\to \EuScript M(X)$ induces a topology on $\EuScript M(X)$. This topology can be described as follows. Markings $m_x\in\EuScript M_x$ and $m_y\in\EuScript M_y$ are close if and only if $x$ is close to $y$ and there exists an orbit conjugacy $\xi_{xy}\colon M_x\to M_y$ between $g_x^t$ and $g_y^t$ which is $C^0$ close to identity (in a chart) and takes the marking $m_x$ to the marking $m_y$. Note that, by structural stability, for any $y$ sufficiently close to $x$ there exist an orbit conjugacy $\xi_{xy}$. Moreover, among orbit conjugacies $C^0$ close to identity, $\xi_{xy}$ is unique up to homotopy that moves points a short distance along the orbits (see~\cite[Section~5]{PSW} for the details and discussion of this uniqueness property). It immediately follows that the map
$$
\EuScript M(p)\colon\EuScript M(X)\to X
$$
which takes $m\in\EuScript M_x\subset\EuScript M(X)$ to its base-point $x$ is a covering map. Because $X$ is simply connected the restriction of $\EuScript M(p)$ to each connected component of $\EuScript M(X)$ is a one sheeted covering. Denote by
$
\bar{\EuScript M}(X)
$ 
 the connected component containing the trivial self marking of $g^t$ in $\EuScript M(X)$. Let 
$$
\bar{\mathcal F}(X)=(mark)^{-1} (\bar{\EuScript M}(X))
$$
and let $\bar F$ be the preimage of the trivial self marking, \ie
$$
\bar F=\{\varphi\in F: \varphi\; \textup{induces trivial self marking of}\; g^t\}.
$$
Clearly $\bar{\mathcal F}(X)$ is open and closed subbundle of $\mathcal F(X)$ and, using the third claim above, it is straightforward to verify that $\bar{\mathcal F}(X)\to X$ is a principal fiber bundle with structure group $\bar F$.

By the definition of $\varphi \in\bar F$, we have $\varphi(\gamma)=\gamma$ for all orbits $\gamma$. 
With this information on the structure group $\bar F$  we can show that $\bar F$ is contractible and then deduce that bundle $p\colon E\to X$ is topologically trivial. The technical conditions P3 and P4 guarantee that this last step can be carried out in exactly the same way as in the proof of Theorem~\ref{thm_anosov_flows}.
\end{proof}

\begin{proof}[Proof of Theorem~\ref{thm_example_flow}]
Let $\T^n\to E\stackrel{p}{\to}\S^d$ be the smoothly non-trivial fiber bundle constructed in Section~\ref{sec_bundle_cnstr} and let $f\colon E\to E$ be the fiberwise Anosov diffeomorphism constructed in Section~\ref{sec_const} for the proof of  Theorem~\ref{thm_example}. Let $E_f$ be the mapping torus of $f$, \ie
$$
E_f=E\times[0,1]/(x,1)\sim (f(x),0).
$$
Then $E_f$ is the total space of the fiber bundle $\T^n_f\to E_f\stackrel{p_f}{\to}\S^d$ whose fibers $\T^n_{f,x}$, $x\in \S^d$, are the mapping tori of $f|_{\T^n_x}$, $x\in\S^d$. It is easy to check that the suspension flow on $E_f$ is a fiberwise Anosov flow which satisfies the assumptions of Theorem~\ref{thm_anosov_flows}. To complete the proof we will show that $p_f\colon E_f\to\S^d$ is not smoothy trivial.

The fundamental group of the mapping torus fiber $\T^n_{f,x}$ is $\Z_{\;f_*}\!\!\ltimes\Z^n$, where $f_*\colon \Z^n\to \Z^n$ is the automorphism of $\pi_1(\T^n)=\Z^n$ induced by $f|_{\T^n_x}$. The inclusion $\T^n_{f,x}\subset E_f$ induces a canonical isomorphism $\pi_1(\T^n_{f,x})\to\pi_1(E_f)$. Hence, both $\pi_1(E_f)$ and $\pi_1(\S^d\times\T^n_f)$ are canonically identified with $\Z_{\;f_*}\!\!\ltimes\Z^n$.

Now, assume to the contrary that there exists a fiber preserving diffeomorphism $\varphi\colon E_f\to \S^d\times\T^n_f$. For both, $E_f$ and $\S^d\times\T^n_f$, consider covers that correspond to the normal subgroup $\Z^n\subset\Z_{\;f_*}\!\!\ltimes\Z^n$, Clearly, the covering spaces are $E\times\R$ and $\S^d\times\T^n\times \R$, respectively.

\begin{claim} Diffeomorphism $\varphi\colon E_f\to \S^d\times\T^n_f$ admits a lifting $\tilde\varphi\colon E\times\R\to \S^d\times\T^n\times\R$.
\end{claim}
\begin{proof}
Using the fact that $f_*$ is hyperbolic one can check that $\Z^n\subset\Z_{\;f_*}\!\!\ltimes\Z^n$ is the commutator subgroup. It follows that the subgroup $\Z^n$ is invariant under all automorphisms of $\Z_{\;f_*}\!\!\ltimes\Z^n$. Hence $\varphi_*(\Z^n)=\Z^n$. Therefore, by the lifting criterion, $\varphi$ admits a lifting to the covers that correspond to $\Z^n$. (Recall that both $\pi_1(E_f)$ and $\pi_1(\S^d\times\T^n_f)$ are canonically identified with $\Z_{\;f_*}\!\!\ltimes\Z^n$.)
\end{proof}
Note that $\tilde\varphi$ is a smooth trivialization of the bundle $E\times\R\to\S^d$. We cut $\S^d\times\T^n_f$ along $\tilde\varphi(E\times\{0\})$. For sufficiently large $t$ the submanifold $\S^d\times\T^n\times\{t\}$ is disjoint with $\tilde\varphi(E\times\{0\})$ and by cutting again along $\S^d\times\T^n\times\{t\}$ we obtain a cobordism $W^{d+n+1}$ such that $\partial^+W^{d+n+1}=E$ and $\partial^-W^{d+n+1}=\S^d\times\T^n$. It is easy to check that $W^{d+n+1}$ is an $h$-cobordism. Since the Whitehead group $Wh(\pi_1(\S^d\times\T^n))$ vanishes the $s$-cobordism theorem applies (by construction in Section~\ref{sec_bundle_cnstr}, $d+n\ge 7$) and we conclude that $E$ is diffeomorphic to $\S^d\times\T^n$. Moreover, $W^{d+n+1}$ smoothly fibers over $\S^d$ with structure group $\textup{Diff}(\T^n\times [0,1], \T^n\times\{0\})$, \ie its structure group is the group of all smooth pseudo-isotopies $\cP(\T^n)$ of the torus $\T^n$. This bundle must of course be trivial over each of the two hemispheres  $\DD^d_\pm$ of $\S^d$. We proceed to use this fact to construct a weakly smooth trivialization of the bundle $p\colon E\to\S^d$. (Recall the definition preceding Addendum~\ref{add_weak_trivial}.) This will contradict Addendum~\ref{add_weak_trivial} according to which $p\colon E\to\S^d$ is not weakly smoothly trivial. And this contradiction will finish the proof of Theorem~\ref{thm_example_flow}. So it remains to construct $g$.

We obtain $g|_{\DD^d_-\times\T^n}$ from the trivialization of the bundle 
$$
\T^n\times[0,1]\to W^{d+n+1}|_{\DD^d_-}\to \DD^d_-
$$
mentioned above. Then we apply the $s$-cobordism theorem to the $h$-cobordism $(W^{d+n+1}|_{\DD^d_+},\DD^d_+\times\T^n)$ to extend the product structure on $W^{d+n+1}|_{\DD^d_+\cap\DD^d_-}$ to all of $W^{d+n+1}|_{\DD^d_+}$. Note that this extension need not be fiber preserving; \ie $g$ needn't be a smooth bundle trivialization. But it is, at least, a weak smooth trivialization.
\end{proof}

\section{Proof of Theorem~\ref{thm_2}}\label{sec_thm2}
\label{sec_thm_2}
Let $M$ be a smooth closed $n$-dimensional manifold and let $f\colon M\to M$ be a diffeomorphism. The differential map $Df\colon TM\to TM$ is linear when restricted to fibers and, hence, induces a bundle self map of $SM$ which we still denote by $Df$. In this way the group $\textup{Diff}(M)$ of all self diffeomorphisms of $M$ also acts on $SM$ and, hence,  embeds into $\textup{Diff}(SM)$.

\subsection{Strategy of the proof}\label{sec_strategy}
We will fix a closed orientable real hyperbolic manifold $M$, whose dimension $n$ is an odd integer and $n\gg 2p-3$ (relative to Igusa's stable range). Our goal is to find a map $\varphi\colon\S^{2p-4}\to \textup{Diff}(M)$ such that the induced map
$$
D\varphi\colon\S^{2p-4}\to \textup{Diff}(SM),
$$
when composed with the natural inclusion $\sigma\colon\textup{Diff}(SM)\hookrightarrow\textup{Top}(SM)$, represents a  non-zero element in $\pi_{2p-4}(\textup{Top}(SM))$. Applying the clutching construction  to $\alpha$ yields the smooth bundle $M\to E\to \S^{2p-3}$ posited in Theorem~\ref{thm_2}. (The definition of clutching was given in Section~\ref{sec_non_Anosov}.) Then notice that
$$
SM\to SE\to\S^{2p-3}
$$
is the result of the clutching construction applied to $D\varphi$. This bundle is not topologically trivial if and only if the class of the composition $\sigma\circ D\varphi$ in the group $\pi_{2p-4}(\textup{Top}(SM))$ is not zero. So it remains to do the construction.

\subsection{The clutching map}
\label{sec_52}
We will use $\cP^s(\cdot)$ and $\cP(\cdot)$ to denote the smooth and topological pseudo-isotopy functors, respectively.

Identify $S^1\times\DD^{n-1}$ with a closed tubular neighborhood of an essential simple closed curve in $M$ which represents an indivisible element of $\pi_1(M)$. Abbreviate $S^1\times\DD^{n-1}$ by $T\subset M$ and note that $\partial T=S^1\times\S^{n-2}$. Denote the ``top'' of a pseudo-isotopy $\cP(T), \cP(M)$, {\it etc.}, by $f_1$, \ie $f_1\colon M\to M$ is $f|_{M\times\{1\}}$. This gives a map $\cP(T)\to\textup{Top}(T)$. 

Consider the map $v\colon\cP^s(T)\to\textup{Top}(SM)$ defined by the following composition
\begin{equation}\label{eq_ii}
\cP^s(T)\stackrel{\iota}{\to}\cP^s(M)\stackrel{\textup{top}}{\longrightarrow}\textup{Diff}(M)\stackrel{D}{\to}\textup{Diff}(SM)\hookrightarrow\textup{Top}(SM),
\end{equation}
where $\iota$ is induced by the inclusion $T\hookrightarrow M$ (\ie $\iota(f)$ coincides with $f$ on $T$ and is identity on $M\backslash T$). We will show that $\pi_{2p-4}(v)(\pi_{2p-4}(\cP^s(T)))$ is a non-trivial group. Then Theorem~\ref{thm_2} follows as discussed in the  Section~\ref{sec_strategy} above.

Since $T$ is parallelizable, we can (and do) fix an identification of $ST$ with $T\times\S^{n-1}$ as bundles over $T$. This allows us to define a map
$$
\widehat{\,\,\,}\colon \cP(T)\to\cP(SM)
$$
as follows
$$
\widehat{f}|_{(SM\backslash S(Int(T)))\times[0,1]}=id
$$
and
$$
\widehat{f}|_{T\times \S^{n-1}\times[0,1]}=f\times id.
$$
Note that $f\mapsto\widehat f$ is a group homomorphism.  

Denote by $u$ the composite map
\begin{equation}\label{eq_i}
\cP^s(T)\hookrightarrow\cP(T)\stackrel{\widehat{\,\,\,}}{\to}\cP(SM)\stackrel{\textup{top}}{\longrightarrow}\textup{Top}(SM).
\end{equation}

Let $\Z_p^\infty$ denote the countably infinite direct sum of copies of the cyclic group $\Z_p$ of order $p$.
\begin{lemma}\label{lemma_1_thm2} There exist a subgroup $\Z_p^\infty\subset \pi_{2p-4}(\cP^s(T))$ which maps monomorphically into $\pi_{2p-4}\textup{Top}(SM)$ under $u_*=\pi_{2p-4}(u)$, \ie under the homomorphism induced by $u$.
\end{lemma}
\begin{cor}
 The image $v_*(\Z_p^\infty)$ in $\pi_{2p-4}(\textup{Top}(SM))$ is not finitely generated and, hence, is non-trivial.
\end{cor}
\begin{proof}
First notice that each of the two compositions, (\ref{eq_ii}) and~(\ref{eq_i}), send $f\in\cP^s(T)$ to two maps $v(f)\colon SM\to SM$ and $u(f) \colon SM\to SM$ that cover the same map $\textup{top}(\iota(f))$of $M$.  Hence the map
$$
w(f)\stackrel{\mathrm{def}}{=}u(f)\circ v(f)^{-1}
$$
covers $id_M$ and, therefore, is equivalent to a smooth map $\varphi\colon M\to GL(n,\R)$ such that $\varphi(x)=id_{\R^n}$ when $x\notin T$.

Notice that
$$
u_*=w_*+v_*,
$$
where maps $u_*, w_*$ and $v_*$ are the functorially induced group homomorphisms $\pi_{2p-4}\cP^s(T)\to\pi_{2p-4}\textup{Top}(SM)$.

So if $w_*(\Z_p^\infty)$ and $v_*(\Z_p^\infty)$ were both finitely generated, then, $u_*(\Z_p^\infty)$ would also be finitely generated. Since it is not, by Lemma~\ref{lemma_1_thm2}, we conclude that either image $w_*(\Z_p^\infty)$ or image $v_*(\Z_p^\infty)$ is not finitely generated. We will presently show that  $w_*(\Z_p^\infty)$ is finitely generated and thus conclude that  $v_*(\Z_p^\infty)$ is not finitely generated proving the Corollary.  

For this consider the bundle
$$
GL(n,\R)\to Aut(TM)\to M,
$$
associated to the tangent bundle $\R^n\to TM\to M$ in the sense of Steenrod~\cite{Steen}. 
The fiber of this bundle is the the space of all invertible linear transformations of the fiber $T_xM$, the tangent space at $x\in M$. (Caveat: this is not the principal frame bundle associated to $TM$.) Let $\Gamma$ be the space of all smooth cross-sections to this bundle. Note that there is a canonical map
$$
\Gamma\to\textup{Diff}(SM)
$$
and, as we observed above, $w$ factors through the composite map
$$
\Gamma_0\to\textup{Diff}(SM)\to\textup{Top}(SM),
$$
where $\Gamma_0$ is the subspace of $\Gamma$ consisting of all cross-sections which map points outside of $T$ to $id$ in the fiber over that point. (Indeed, every element in $\Gamma_0$ is equivalent to a smooth map $\varphi\colon M\to GL(n,\R)$ such that $\varphi(x)=id_{\R^n}$ when $x\notin T$.) Hence, to prove that the image $w_*(\Z_p^\infty)$ is finitely generated, it suffices to show the following claim.
\begin{claim}
Let $k=2p-4$. Then the group $\pi_k\Gamma_0$ is finitely generated.
\end{claim}
Because $T$ is parallelizable, $\Gamma_0$ is clearly homeomorphic to the space of all smooth maps
$$
(T,\partial T)\to (GL(n,\R), id)
$$
and therefore
\begin{multline*}
\pi_k\Gamma_0\simeq [\DD^k\times T,\partial(\DD^k\times T); O(n,\R), id]\\
=[S^1\times\DD^{n+k-1},\partial(S^1\times\DD^{n+k-1}); O(n,\R), id]\\
=[\S^{n+k}\vee\S^{n+k-1}, \mbox{wedge pt.}; O(n,\R), id]\\
=\pi_{n+k}O(n,\R)\oplus\pi_{n+k-1}O(n,\R).
\end{multline*}
Hence, by Serre's thesis, this group is finitely generated.
\end{proof}

\subsection{Proof of Lemma~\ref{lemma_1_thm2}}
Let $\ccP^s(\cdot)$ and $\ccP(\cdot)$ denote the stable smooth and topological pseudo-isotopy (homotopy) functors, respectively; cf.~\cite[\textsection 1]{Hat}. The map $\widehat{\,\,\,}\colon P(T)\to P(SM)$, defined in Section~\ref{sec_52}, induces a map $\ccP(T)\to\ccP(SM)$ which we still denote by $\widehat{\,\,\,}$. (Recall that the are working in the stable dimension range.)

The argument given in~\cite[pp. 1419-20]{FO3} for proving Theorem~D
 of that paper carries over, when the manifold $N$ in~\cite{FO3} is replaced by $SM$, to prove Lemma~\ref{lemma_1_thm2} provided we can find a subgroup $\Z_p^\infty$ of $\pi_k(\ccP^s(T))$, $k=2p-4$, such that the composite map
 \begin{equation}\label{eq_star}
 \Z_p^\infty\subset\pi_k(\ccP^s(T))\to\pi_k(\ccP(T))\stackrel{\widehat{\,\,\,}_*}{\longrightarrow}\pi_k(\ccP(SM)) \to H_0(\Z_2; \pi_k(\ccP(SM)))
 \end{equation}
 is one-to-one. For this reduction we use the fact that the homomorphism induced by inclusion
 $$
 \pi_k(\ccP^s(T))\to\pi_k(\ccP(T))
 $$
 is an isomorphism modulo the Serre class of finitely generated abelian groups; cf.~\cite[Lemma 4.1]{FJ87}. To find such $\Z_p^\infty$ consider the following commutative diagram
$$
\xymatrix{
\Z_p^\infty    \ar^{\subset\,\,\,\,\,\,\,\,}[r]   &  \pi_k(\ccP^s(T))  \ar[r]   &  \pi_k(\ccP(T)) \ar[rr]    & &   H_0(\Z_2;\pi_k(\ccP(T)))\\
\Z_p^\infty    \ar^{\times 2}[u] \ar^{\subset\,\,\,\,\,\,\,\,}[r]   &  \pi_k(\ccP^s(T))  \ar^{\times 2}[u] \ar[r]   &  \pi_k(\ccP(T)) \ar^{\times 2}[u] \ar[r] & \pi_k(\ccP(ST)) \ar[ul]^{q_*}\ar[r] \ar[d]^{i_*}   &    H_0(\Z_2;\pi_k(\ccP(ST))) \ar[u]^{H_0(q_*)}\ar[d]_{H_0(i_*)}\\
&&& \pi_k(\ccP(SM)) \ar[r] & H_0(\Z_2;\pi_k(\ccP(SM)))
}
$$
{\it Comments on the diagram:}
\begin{enumerate}
 \item The letter $q$ denotes the bundle projections in both of the bundles $ST\to T$ and $SM\to M$. (The latter one does not appear in the diagram.)
 \item The triangle in this diagram commutes because of~\cite[Proposition on p.~18]{Hat}. (Recall that the dimension $n$ is an odd integer.)
 \item The subgroup $\Z_p^\infty$ in $\pi_k(\ccP^s(T))$ comes from~\cite[Proposition 4.5]{FO2} using the remarks on the last three lines of page 1420 of~\cite{FO3} and the first three lines on the next page. By its construction $\Z_p^\infty$ maps in a one-to-one fashion into $H_0(\Z_2; \pi_k(\ccP(N)))$ via the composition of the maps given on the top line of the diagram.
 \item\label{comment_4} The functorially induced maps
 $$
 \begin{multlined}
 q_*\colon \pi_k(\ccP^s(ST))\to\pi_k(\ccP^s(T)),\\
 q_*\colon \pi_k(\ccP^s(SM))\to\pi_k(\ccP^s(M))
 \end{multlined}
 $$
 are $\Z_2$-modules isomorphisms because of~\cite[Corollary 5.2 and Lemma on p. 18]{Hat} and the facts that $k<n-1$ and $n-1$ is even. The action of $\Z_2$ on $\ccP(T)$ and $\ccP(M)$ are induced by ``turning a pseudo-isotopy upside down.'' (See~\cite[pp. 6-7]{Hat} for more details.) Consequently, the functorially induced map
 $$
 H_0(q_*)\colon H_0(\Z_2; \pi_k(\ccP^s(ST))\to H_0(\Z_2; \pi_k(\ccP^s(T))
 $$
 is an isomorphism.
 \item The inclusion maps of $T$ into $M$ and $ST$ into $SM$ are both denoted by $i$. Then
$$
 \begin{multlined}
 i_*\colon \pi_k(\ccP^s(T))\to\pi_k(\ccP^s(M)),\\
 i_*\colon \pi_k(\ccP^s(ST))\to\pi_k(\ccP^s(SM))
\end{multlined}
$$
denote the functorially induced group homomorphisms which are both $\Z_2$-module maps. Furthermore,
$$
H_0(i_*)\colon H_0(\Z_2;\pi_k(\ccP^s(ST)))\to H_0(\Z_2;\pi_k(\ccP^s(SM)))
$$
is the group homomorphism induced functorially by $ i_*$.
\item The main result of~\cite{FJ87} shows that $i_*$ is an isomorphism onto a direct summand of $\pi_k(\ccP(M))$ as $\Z_2$-modules. Hence, so is
$
i_*\colon \pi_k(\ccP^s(ST))\to\pi_k(\ccP^s(SM))
$
by comment~\ref{comment_4} above (note that $q\circ i=i\circ q$). Consequently, $H_0(i_*)$ is monic. 
\end{enumerate}
Now ``chasing the diagram'' verifies that the composite map~(\ref{eq_star}) is one-to-one completing the proof of Lemma~\ref{lemma_1_thm2}, and thus also Theorem~\ref{thm_2}.


\section{Proof of Theorem~\ref{thm_anosov_nil}}\label{sec_thm_anosov}
Denote by $G(M)$ the space of self homotopy equivalences of $M$ and by $G_0(M)\subset G(M)$ the connected component of $id_M$. Also denote by $\textup{Top}(M)$ the group of self homeomorphisms of $M$ and by $\textup{Top}_0(M)$ subgroup of $\textup{Top}(M)$ which consists of homeomorphisms homotopic to $id_M$ (through a path of maps). The proof of Theorem~\ref{thm_anosov_nil} is based on the following lemma. 
\begin{lemma}\label{lemma_anosov_nil}
Let $M$ be a (closed) infranilmanifold.
 Then $\textup{Top}_0(M)$ contains a connected topological subgroup $S$ satisfying the following properties.
\begin{enumerate}
\item The group $S$ is a connected Lie group.
\item The composition $\sigma\colon S\to G_0(M)$ of the inclusion maps $S\subset \textup{Top}_0(M)$ and $\textup{Top}_0(M)\subset G_0(M)$ induces an isomorphism
$$
\sigma_*\colon \pi_i(S)\to \pi_iG_0(M)
$$
for all $i$.
\item If $L\colon M\to M$ is an  affine Anosov diffeomorphism of $M$, then 
$$F\stackrel{\mbox{def}}{=}\{\varphi\in\textup{Top}_0(M) : \varphi\circ L=L\circ\varphi\}
$$
is a subgroup of $S$.
\end{enumerate}
\end{lemma}
We proceed with the proof of Theorem~\ref{thm_anosov_nil} assuming Lemma~\ref{lemma_anosov_nil}.  Let $q\colon E\to M$ be a fiber homotopy trivialization of the Anosov bundle
$$
M\to E\stackrel{p}{\to}X.
$$ 
As in the proof of Theorem~\ref{thm_anosov_flows} we identify $M$ with a particular fiber $M_{x_0}=p^{-1}(x_0)$. Recall that we can assume without loss of generality that $q|_{M}\colon M\to M$ is homotopic to identity; moreover, by performing a homotopy in a chart we can (and do) assume that $q|_{M}=id_M$. Also recall that the bundle $M\to E\to X$ admits a fiberwise Anosov diffeomorphism. Recall that we denote by $f_x$, $x\in X$, the Anosov diffeomorphisms of the fibers $M_x$, $x\in X$. We will write  $f\colon M\to M$ for the restriction $f_{x_0}\colon M\to M$ of the Anosov bundle map to the fiber $M=M_{x_0}$.

For each $x\in X$ consider the space $F_x$ of all homeomorphisms $\varphi \colon M\to M_x$ which are conjugacies between $f$ and $f_x$ and which also satisfy the property ``$q\circ\varphi\colon M\to M$ is homotopic to $id_M$." It follows from the uniqueness part of the structural stability theorem that each $F_x$ is a discrete topological space. 

Let
$$
\mathcal F(X)=\bigsqcup_{x\in X}F_x.
$$
Also let $\mathcal F(p)\colon \mathcal F(X)\to X$ be the map that sends the fibers $F_x$ to their base-points, $x\in X$. By an argument similar to that given in the proof of Theorem~\ref{thm_anosov_flows}, $\mathcal F(p)\colon \mathcal F(X)\to X$ is a covering of $X$. Moreover, it is a regular covering space; \ie a principal bundle with structure group 
$$F_{x_0}=\{\varphi\in\textup{Top}_0(M) : \varphi\circ f=f\circ\varphi\}.$$ 
By work of Franks and Manning~\cite{Fr, Mann}, the Anosov diffeomorphism $f\colon M\to M$ is conjugate to an affine automorphism $L\colon M\to M$ via a conjugacy homotopic to $id_M$. This conjugacy gives an isomorphism between topological groups $F_{x_0}$ and $F$. Thus we can (and do) identify $F_{x_0}$ and $F$.

Notice that the associated $M$-bundle $\mathcal F(X)\times_FM$ is isomorphic to $E$ via $(\varphi, y)\mapsto \varphi(y)$, as $\textup{Top}(M)$-bundles. Hence we have reduced the structure group of the bundle $p\colon E\to X$ from $\textup{Top}(M)$ to the subgroup $F$ and, therefore, to $S$, where $S$ is the topological subgroup of $\textup{Top}_0(M)$ posited in Lemma~\ref{lemma_anosov_nil}. Consequently, it suffice to show that the principal $S$-bundle ${\mathcal F(X)\times_FS\to X}$ associated to $p\colon E\to X$ is trivial. 

Recall that the total space $\mathcal F(X)\times_FS$ is the quotient of $\mathcal F(X)\times S$ by the action of $F$, given by $(\varphi,s)\mapsto(\varphi\circ \psi,\psi^{-1}\circ s)$, $\psi\in F$.
Therefore the map $\eta\colon\mathcal F(X)\times_FS\to G_0(M)$ given by
$$
(\varphi,s)\stackrel{\eta}{\mapsto} q\circ \varphi\circ s 
$$
is well defined. Notice that because $q|_M=id_M$, map $\eta$ restricted to the fiber over $x_0$ is the composite map
$$
\sigma\colon S\to\textup{Top}_0(M)\to G_0(M).
$$

Because $G_0(M)$ has the homotopy type of a CW complex, which is a consequence of~\cite[Theorem 1]{M59}, map $\sigma$ is a homotopy equivalence by property~2 of Lemma~\ref{lemma_anosov_nil} and, hence, has a homotopy inverse $\sigma^{-1}$. Then the map $\sigma^{-1}\circ\eta\colon\mathcal F(X)\times_FS\to S$ is a homotopy equivalence when restricted to the fibers, \ie the bundle $\mathcal F(X)\times_FS\to X$ is fiber homotopically trivial. Hence it has a cross-section and consequently is a trivial $S$-bundle.

To finish the proof of Theorem~\ref{thm_anosov_nil} it remains to establish Lemma~\ref{lemma_anosov_nil}.
\begin{proof}[Proof of Lemma~\ref{lemma_anosov_nil}]
This is a consequence of the construction and arguments made in Appendix~C of~\cite{FG2}. In particular, $S$ is $\EuScript N_0^G$, \ie the connected component containing identity element of the Lie group $\N^G$ from Fact~23 of Appendix~C. We recall its definition.

The universal cover of $M$ is a simply connected nilpotent Lie group $N$ and $\pi_1(M)$,  considered as the group of deck transformations of $N$, is a discrete subgroup of $N\rtimes G$, where $G$ is a finite group of automorphisms of $N$ which maps monomorphically into $Out(N)$. Also $\pi_1(M)$ projects epimorphically onto $G$. Let $\Gamma=\pi_1(M)\cap N$, \ie the subgroup of $\pi_1(M)$ that acts by pure translations. The homogeneous space $\widehat M=N/\Gamma$ is a nilmanifold, which is a regular finite-sheeted cover of $M$ whose group of deck transformations can be identified with $G$. The natural action of $N$ on $\widehat M$ gives a representation
$$
\hat\rho\colon N\to\textup{Top}_0(\widehat M),
$$
whose kernel is $\mathcal Z(N)\cap\Gamma$, where $\mathcal Z(N)$ denotes the center of $N$. The image of this representation is the nilpotent Lie group $\N$, \ie
$$
\N=N/\mathcal Z(\Gamma).
$$
(It is a consequence of Mal$'$cev rigidity that $\mathcal Z(\Gamma)=\mathcal Z(N)\cap\Gamma$, see~\cite{Malcev}.)
Now $G$ acts on $\N$ by conjugation. Let $\N^G$ denote the subgroup of $\N$ which is fixed by $G$. The group $\N^G$ maps isomorphically onto a topological subgroup of $\textup{Top}_0(M)$ which we also denote by $\N^G$. Then the connected Lie group $S$ posited in Lemma is the connected component of identity element $\N^G_0$. The argument on the last 10 lines of Appendix~C (crucially using~\cite{W}) shows that $F$ is a subgroup of $S$, \ie verifies property 3 of Lemma.

So it remains to verify property 2. Recall~(\ref{eq_homotopy_homotopy})
 $$
 \pi_n(G_0(M))=
 \begin{cases}
  \mathcal Z(\pi_1(M)), & \mbox{if}\,\, n=1\\
  0, & \mbox{if}\,\, n\ge 2
 \end{cases}
 $$
 So we need a similar calculation for $\pi_n(S)$. 
 
 The group $S$ is a connected nilpotent Lie group, hence $\pi_n(S)=0$ for $n\ge 2$. To calculate $\pi_1(S)$, we use that
 $$
 \rho\colon N^{G}\to S
 $$
 is a covering space. To see this apply the functor $(\cdot)^G$ to the following (covering space) exact sequence
 $$
 1\to\mathcal Z(\Gamma)\to N\to \N\to 1,
 $$
 which yields the half exact sequence 
 $$
 1\to \mathcal Z(\Gamma)^G\to N^G\to \N^G\to  .
 $$
 Then, using a Lie algebra argument, we see that the image of $N^G$ in $\N^G$ is $\N^G_0=S$, since $N^G$ is connected by Fact~22 of~\cite[Appendix~C]{FG2}. Moreover, $N^G$ is contractible since it is a closed connected subgroup of the simply connected nilpotent Lie group $N$. Consequently,
 $$
 \pi_1(S)=\ker(\rho)=\mathcal Z(\Gamma)^G.
 $$
 \begin{claim} Group $\mathcal Z(\Gamma)^G$ is isomorphic to $\mathcal Z(\pi_1(M))$.
 \end{claim}
 \begin{proof} 
 An element $v$ is in $\mathcal Z(\Gamma)^G$ if and only if the following two conditions are satisfied.
 \begin{enumerate}
 \item $v\in\mathcal Z(\Gamma)$;
 \item $\forall h\in G$, $h(v)=v$.
 \end{enumerate}
 Then, the inclusion $\mathcal Z(\Gamma)^G\subset \mathcal Z(\pi_1(M))$ follows from the following identity
 \begin{equation}
 \label{eq_identity}
 (h,u)(id_N, v)(h,u)^{-1}=(id_N,uh(v)u^{-1})
 \end{equation}
 valid for all $h\in G$ and $u,v\in N$.
 
 To check the opposite inclusion $\mathcal Z(\pi_1(M))\subset\mathcal Z(\Gamma)^G$ take $(g,v)\in \mathcal Z(\pi_1(M))$. Then
 \begin{equation}\label{eq_identity2}
 uh(v)=vg(u)
 \end{equation}
 for all $(h,u)\in\mathcal Z(\pi_1(M))$. By choosing $h=id_N$ we obtain $g(u)=v^{-1}uv$ for all $u\in\Gamma$, and, hence, $g$ is an inner automorphism of $N$. But $G$ maps monomorphically into $Out(N)$; therefore, $g=id_N$ and $v\in\mathcal Z(\Gamma)$. Finally, to see that $(id_N,v)\in \mathcal Z(\pi_1(M))$ satisfies the condition 2 as well, we use the fact that the composition $\pi_1(M)\subset G\ltimes N\to G$ is an epimorphism in conjunction with~(\ref{eq_identity}) and~(\ref{eq_identity2}).
  \end{proof}
 
 Therefore, the above claim implies $\pi_1(S)\simeq\pi_1(G_0(M))$.
 
 
 
 It remains to show that $\sigma_*\colon\pi_1(S)\to\pi_1(G_0(M))$ is an isomorphism. However, $\pi_1(S)\subset\mathcal Z(\Gamma)$ and, hence, $\pi_1(S)$ is a finitely generated abelian group. Therefore it suffices to show that $\sigma_*$ is an epimorphism, but this is a consequence of Fact~22 of~\cite[Appendix~C]{FG2}.
\end{proof}

To prove Addeddum~\ref{add_23} we need the following modification of Lemma~\ref{lemma_anosov_nil}.
\begin{lemma}
 \label{lemma_anosov_nil2}
 Let $M$ be a (closed) nilmanifold. Then $\textup{Top}(M)$ contains a topological subgroup $S$ satisfying the following properties.
\begin{enumerate}
\item The group $S$ is a Lie group.
\item The composition $\sigma\colon S\to G(M)$ of the inclusion maps $S\subset \textup{Top}(M)$ and $\textup{Top}(M)\subset G(M)$ induces an isomorphism
$$
\sigma_*\colon \pi_i(S)\to \pi_iG(M)
$$
for all $i\ge 0$.
\item If $L\colon M\to M$ is an  affine Anosov diffeomorphism of $M$, then 
$$F\stackrel{\mbox{def}}{=}\{\varphi\in\textup{Top}(M) : \varphi\circ L=L\circ\varphi\}
$$ 
is a subgroup of $S$.
\end{enumerate}
\end{lemma}

We proceed with the proof of Addeddum~\ref{add_23} assuming Lemma~\ref{lemma_anosov_nil2}. 
\begin{proof}[Proof of Addeddum~\ref{add_23}]
By the argument used to prove Theorem~\ref{thm_anosov_nil} we can identify the nilmanifolds $M_i=N_i/\Gamma_i$, $i=1, 2$, with fibers $M_{i,x_0}$ for a fixed base point $x_0\in X$. By the argument used in the proof of Theorem~\ref{thm_anosov_nil} the Anosov diffeomorphisms $f_i\colon M_i\to M_i$ can be identified with affine Anosov diffeomorphisms $L_i$. 
We proceed as in the proof of Theorem~\ref{thm_anosov_nil} separately for each bundle $p_i\colon E_i\to X$. For $x\in X$ consider the fibers $F_{i,x}$ that consist of all homeomorphisms $\varphi \colon M_i\to M_{i,x}$ which are conjugacies between $f_i$ and $f_{i,x}$. (Note we drop the constraint that involves fiber homotopy equivalence.) By taking the union of these fibers we still obtain principal bundles
$$
\mathcal F(p_i)\colon \mathcal F_i(X)\to X,
$$
which reduce (inside of $\textup{Top}(M_i)$) the structure group of each bundle $p_i\colon E_i\to X$ to the group 
$$
F_i\stackrel{\mbox{def}}{=}\{\varphi\in\textup{Top}(M_i) : \varphi\circ L_i=L_i\circ\varphi\}.
$$

By Mal$'$cev's rigidity theorem, the given homotopy equivalence $M_1=M_{1,x_0}$ to $M_2=M_{2,x_0}$ determines an affine diffeomorphism between $M_1$ and $M_2$ via which we identify $M_1$ and $M_2$ to a common nilmanifold $M$ to which we apply Lemma~\ref{lemma_anosov_nil2}. In this way we reduce the structure group of each bundle, inside $\textup{Top}(M)$, to the Lie group $S$ posited in Lemma~\ref{lemma_anosov_nil2}.

Let $k_i\colon X\to BS$, $i=1,2$, be a continuous map which classifies the bundle $E_i$ in terms of its reduced structure group $S$. Here $BS$ denotes the classifying space for bundles with structure group $S$. Similarly, $BG(M)$ is the classifying space constructed by Stasheff~\cite{St} for homotopy $M$-bundles up to fiber homotopy equivalence, cf.~\cite[pages 3-7]{MM}. The inclusion map $S\subset G(M)$ induces a map
$$
\eta\colon BS\to BG(M),
$$
which corresponds to the forget structure map at the bundle level. Because of property~2 in Lemma~\ref{lemma_anosov_nil2}, $\eta$ is a weak homotopy equivalence and hence induces a bijection
$$
\eta_*\colon [X,BS]\to [X, BG(M)]
$$
between homotopy classes of maps, cf.~\cite[p. 405, Cor. 23]{S66}.

Since the bundles $E_1$ and $E_2$ are fiber homotopically equivalent, the composite maps $\eta\circ k_1$ and $\eta\circ k_2$ are homotopic. Consequently, $k_1$ is homotopic to $k_2$. Let $\sigma\colon BS\to B\textup{Top}(M)$ denote the map induced by inclusion $S\subset \textup{Top}(M)$. Then $\sigma\circ k_1$ is homotopic to $\sigma\circ k_2$ and therefore $E_1$ and $E_2$ are equivalent as $\textup{Top}(M)$-bundles.
\end{proof}

To finish the proof of Addeddum~\ref{add_23} it remains to establish Lemma~\ref{lemma_anosov_nil2}.
\begin{proof}[Proof of Lemma~\ref{lemma_anosov_nil2}]
We start by recollecting some terminology. First recall that $M=N/\Gamma$, where $N$ is a simply connected nilpotent Lie group and $\Gamma\subset N$ is a discrete cocompact subgroup. An affine diffeomorphism of $N$ is the composite of an automorphism $\alpha\in Aut(N)$ with a left translation $L_a\colon N\to N$, where $a\in N$, which is defined by
$$
x\mapsto ax,\;\;\;\textup{for all}\; x\in N.
$$
Likewise, the right translation $R_a\colon N\to N$ and the inner automorphism $I_a\colon N\to N$ are defined by 
$$
x\mapsto xa\;\;\;\textup{and}\;\; x\mapsto axa^{-1},
$$
respectively. Note that $L_a$ induces a self diffeomorphism of the left coset space $N/\Gamma$ (nilmanifold) given by the formula
$$
x\Gamma\mapsto ax\Gamma.
$$
The right translation $R_a$ also induce a self diffeomorphism when $a\in \Gamma$. Denote these diffeomorphisms  again by $L_a$ and $R_a$, respectively.

Mal$'$cev showed that every $\alpha\in Aut(\Gamma)$ induces an automorphism $\bar\alpha \in Aut(N)$, which in turn induces a self diffeomorphism (denoted again by) $\alpha\in\textup{Diff}(N/\Gamma)$ given by the formula
$$
x\Gamma\mapsto \bar\alpha(x)\Gamma.
$$
In particular, if $a\in\Gamma$ then the inner automorphism $I_a$ induces a self diffeomorphism $I_a\in\textup{Diff}(N/\Gamma)$. An {\it affine diffeomorphism} of $N/\Gamma$ is by definition the composite of maps of the form $\alpha$ and $L_a$ in $\textup{Diff}(N/\Gamma)$. Note that $a\mapsto L_a$ induces an homomorphism of $N$ onto a closed subgroup of $\textup{Top}(N/\Gamma)$ whose kernel is $\mathcal Z(N)\cap\Gamma=\mathcal Z(\Gamma)$; while $Aut(\Gamma)\to \textup{Diff}(N/\Gamma)$ induces an embedding of $Out(\Gamma)$ onto a discrete subgroup of $\textup{Top}(N/\Gamma)$. We note the following identities
$$
\alpha\circ L_a\circ\alpha^{-1}=L_{\alpha(a)};\;\; R_a=L_a\circ I_{a^{-1}}=I_{a^{-1}}\circ L_a,
$$
for all $a\in N$ and $\alpha\in Aut(N)$. From these one easily deduces the following equations
\begin{enumerate}
 \item $Af\!f(N)\simeq Aut(N)\ltimes N$;
 \item $Af\!f(N/\Gamma)\simeq Aut(\Gamma)\ltimes N/\sigma(\Gamma)$,
\end{enumerate}
where $\sigma\colon \Gamma\to Af\!f(N)$ is the monic anti-homomorphism given by $\sigma(a)=R_a$, $a\in\Gamma$. (Here $Af\!f(N)$ and $Af\!f(N/\Gamma)$ denote the closed subgroups of $\textup{Top}(N)$ and $\textup{Top}(N/\Gamma)$, respectively, consisting of all affine diffeomorphisms. Furthermore, $\sigma(\Gamma)$ is a discrete normal subgroup of $Af\!f(N)$.)

Under the canonical projection $Aut(\Gamma)\ltimes N$ to $Aut(\Gamma)$, $\sigma(\Gamma)$ maps onto the group of all inner automorphisms of $\Gamma$, $Inn(\Gamma)$ with kernel $\mathcal Z(\Gamma)$. Hence the canonical exact sequence
$$
1\to N\to Aut(\Gamma)\ltimes N\to Aut(\Gamma)\to 1
$$
yields the following quotient exact sequence
$$
1\to N/\mathcal Z(\Gamma)\to Af\!f(N/\Gamma)\to Out(\Gamma)\to 1
$$
Using this we easily deduce that the inclusion map $Af\!f(N/\Gamma)\subset G(M/\Gamma)$ induces an isomorphism on the homotopy groups $\pi_i$ for $i=0,1,2,\ldots $. We now define $S$ to be this subgroup $Af\!f(N/\Gamma)$ of $\textup{Top}(N/\Gamma)=\textup{Top}(M)$. Hence, we have already verified the properties~1 and~2 asserted in Lemma~\ref{lemma_anosov_nil2}. The remaining property~3 follows immediately from Theorem~2 in Walters' paper~\cite{W}.
\end{proof}
 
\begin{remark}
 We conjecture that Addendum~\ref{add_23} remains true in the more general situation where the fibers of the Anosov bundles $p_i\colon E_i\to X$, $i=1,2$, are only assumed to be infranilmanifolds. The proof given above for Addendum~\ref{add_23} would work in this more general setting provided Lemma~\ref{lemma_anosov_nil2} remained true when $M$ is only assumed to be an infranilmanifold.
\end{remark}


\section{Final remarks}
\subsection{A partially hyperbolic diffeomorphism of $(\S^d\times\T^n)\#\Sigma_\alpha$}\label{sec_61}
Here we explain that the fiberwise Anosov diffeomorphism $f\colon E_\alpha\to E_\alpha$ constructed in Section~\ref{section_anosov_bundles} can be viewed as a partially hyperbolic diffeomorphism and point out certain properties of this diffeomorphism.

First we need to recall some definitions from partially hyperbolic dynamics. Given a closed Riemannian manifold $N$ a diffeomorphism $f\colon N\to N$ is called {\it partially hyperbolic} if there
exists a continuous $Df$-invariant non-trivial splitting of the tangent bundle
$TN=E^s_f\oplus E^c_f\oplus E^u_f$ and positive constants
$\nu<\mu_-<\mu_+<\lambda$, $\nu<1<\lambda$,
and $C>0$ such that for $n>0$
 \begin{multline*}
{\;\;\;\;\;\;\;\;\;\;\;\;\;\;\;\;\;\;\;\;\;\;\;\;\;\;\|D(f^n)v\|\le C\nu^n\|v\|,\;\;\; v\in E^s_f,}\\
\shoveleft{\;\;\;\;\;\;\;\frac1C\mu_{-}^n\|v\|\le\|D(f^n)v\|\le C\mu_{+}^n\|v\|,\;\;\; v\in E_{f}^c,}\\
\shoveleft{\;\;\;\;\;\;\;\frac1C\lambda^n\|v\|\le\|D(f^n)v\|,\;\;\; v\in E_{f}^u.} \hfill
\end{multline*}
It is well known that the distributions $E^s_f$ and $E_f^u$ integrate uniquely to
foliations $\W_f^s$ and $\W_f^u$. If the distribution $E^c_f$ also integrates to an $f$-invariant foliation $\W_f^c$ then $f$ is called {\it dynamically coherent}. Furthermore, $f$ is called {\it robustly dynamically coherent} if any diffeomorphism sufficiently $C^1$ close to $f$ is also dynamically coherent. 

Using Hirsch-Pugh-Shub structural stability theorem~\cite[Theorem
7.1]{HPS} we establish the following result.
\begin{prop}\label{prop_phd}
 Let $X$ be a closed simply connected manifold and let $M\to E\stackrel{p}{\to}X$ be a smoothly non-trivial bundle. Assume that $f\colon E\to E$ is a fiberwise Anosov diffeomorphism. Then diffeomorphism $f$ is partially hyperbolic and robustly dynamically coherent. Moreover, for any $g$ which is sufficiently $C^1$ close to $f$, the center foliation $\W_g^c$ is not smooth.
\end{prop}
Note that this proposition applies to $f\colon E_\alpha\to E_\alpha$ constructed in Section~\ref{section_anosov_bundles}. Recall that, by construction, the total space $E_\alpha$ is diffeomorphic to $(\S^d\times\T^n)\#\Sigma_\alpha$, where $\Sigma_\alpha$ is a homotopy $(d+n)$-sphere, obtained by clutching using diffeomorphism $\alpha\in\textup{Diff}_{k+d-1}(\DD^{n+d-1},\partial)$.

In general, the center foliation $\W^c_f$ of a dynamically coherent partially hyperbolic diffeomorphism is continuous foliation with smooth leaves. To the best of our knowledge, all previously known examples of partially hyperbolic dynamically coherent diffeomorphisms $f$ have the following property: \\
$(\star)$ {\it There exists a path $f_t, t\in [0,1],$ of partially hyperbolic dynamically coherent diffeomorphism such that $f_0=f$ and the center foliation $\W^c_{f_1}$ is a smooth foliation.}
\begin{conj}
 The partially hyperbolic diffeomorphism $f\colon E_\alpha\to E_\alpha$ constructed in Section~\ref{section_anosov_bundles} does not have property $(\star)$.
\end{conj}
Note that Proposition~\ref{prop_phd} supports this conjecture.
\begin{proof}[Proof of Proposition~\ref{prop_phd}]
Let 
$$
T^{\textup{fib}}E\stackrel{\textup{def}}{=}\bigcup_{x\in X}TM_x
$$
be the bundle of vectors tangent to the fibers of $p\colon E\to X$. We have a $Df$-invariant splitting $T^{\textup{fib}}E=E_f^s\oplus E_f^u$, where $E_f^s$ is exponentially contracting and $E_f^u$ is exponentially expanding. 

Recall that the proof of Theorem~\ref{thm_anosov_sc} yields a topological trivialization
$
h\colon E\to M.
$
Hence we have a fiber preserving homeomorphism $(p,h)\colon E\to X\times M$. 

For each $y\in E$ consider the the preimage $W_f^c(y)\stackrel{\textup{def}}{=}(p,h)^{-1}(X\times\{h(y)\})$. Because of our contruction of $h$, which locally comes from parametric families of homeomorphisms given by structural stability theorem, and smooth dependence of these homeomorphisms on parameter (\ie base point)~\cite[Theorem~A.1]{LMM}, we obtain that, in fact, $\W_f^c(y)$ are smoothly embedded submanifolds of $E$ whose tangent space $E^c_f(y)\stackrel{\textup{def}}{=}T_y\W_f^c(y)$ depends continuously on $y\in E$. Hence $\W_f^c$ is a continuous foliation by smooth compact leaves. Moreover, this foliation is $f$-invariant. This follows from the fact that $\W_f^c$ is the pullback of the ``horizontal'' foliation on $X\times M$ and $(p,h)$ is a conjugacy between $f$ and $id_X\times f|_M$ (recall, that in the proof of Theorem~\ref{thm_anosov_sc} we have identified $M$ with a particular fiber $M_{x_0}$). Therefore the distribution $E_f^c$ is $Df$-invariant. Then, because the dynamics induced by $f$ on the base is $id_X$, one can check that $f$ is partially hyperbolic with respect to the $Df$-invariant splitting $TE=E_f^s\oplus E_f^c\oplus E_f^u$.

Since $\W_f^c$ is a compact foliation, Hirsch-Pugh-Shub structural stability theorem~\cite[Theorem
7.1]{HPS} in conjunction with the result which verifies plaque expansivity hypothesis for applying Hirsch-Pugh-Shub structural stability (see \eg \cite{PSW}) yields that $f$ is robustly dynamically coherent. Now, for any $g$ sufficiently $C^1$ close to $f$ the center foliation $\W_g^c$ is $C^1$ close to $\W_f^c$ and, hence, is transverse to the fibers of $p\colon E\to X$. Therefore $\W_g^c$ defines holonomy homeomorphisms $M_x\to M (=M_{x_0})$. These holonomy homeomorphisms assemble into a topological trivialization of $p\colon E\to X$. Note that if $\W_g^c$ is smooth then holonomies are smooth and the trivialization will be smooth as well. Hence $\W_g^c$ cannot be smooth, because $p\colon E\to X$ is smoothly non-trivial.
\end{proof}

\subsection{Partially hyperbolic setup}\label{sec_64}
It would be very interesting to generalize our results (\eg Theorems~\ref{thm_anosov_sc},~\ref{thm_anosov_nil}, and~\ref{thm_anosov_flows}) to the more general setting of bundles that admit fiberwise partially hyperbolic diffeomorphisms (see Section~\ref{sec_61} for the definition of a partially hyperbolic diffeomorphism).
\begin{conj}
 Let $X$ be a closed manifold or finite simplicial complex and let $p\colon E\to X$ be a fiber homotopically trivial bundle whose fiber $M$ is a closed manifold. Assume that $E$ admits a fiberwise partially hyperbolic diffeomorphism whose center distribution is lower dimensional (1 or 2). Then $p\colon E\to X$ is topologically trivial.
\end{conj}
The above conjecture can be viewed as a possible generalization of Theorem~\ref{thm_anosov_flows}. Theorem~\ref{thm_anosov_flows} depended on the Livshitz Theorem for Anosov flows in a crucial way. For the above conjecture one may try replacing the argument that uses Livshitz Theorem with a similar argument that would use a heat flow along the center foliation instead. 

On the other hand, Theorem~\ref{thm_anosov_sc} does not admit straightforward (conjectural) generalization. This can be seen via the following example. Let $A\colon\T^2\to\T^2$ be an Anosov diffeomorphism, then $f=(A,id_{\S^3})\colon\T^2\times\S^3\to\T^2\times\S^3$ is partially hyperbolic. By multiplying the Hopf fibration by $\T^2$ we can view $\T^2\times\S^3$ as the total space of the fiber bundle $\T^3\to\T^2\times\S^3\to\S^2$. It is clear that $f$ is a fiberwise partially hyperbolic diffeomorphism, whose restrictions to the fibers have the form $A\times id_{S^1}$. Hence, for a generalization of Theorem~\ref{thm_anosov_sc} we need to further restrict the class of fiberwise partially hyperbolic diffeomorphism. One natural class is the following one. A partially hyperbolic diffeomorphism $g\colon N\to N$ is called {\it irreducible} if it verifies the following conditions:
\begin{enumerate}
 \item diffeomorphism $g$ does not fiber over a (topologically) partially hyperbolic (or Anosov) diffeomorphism $\hat g\colon \hat N\to \hat N$ of a lower dimensional manifold $\hat N$; that is, one cannot find a fiber bundle $p\colon N\to \hat N$ and a (topologically) partially hyperbolic (or Anosov) diffeomorphism $\hat g\colon \hat N\to \hat N$ such that $p\circ g=\hat g\circ p$;
 \item if $g'$ is homotopic to $g$ then $g'$ also verifies 1;
 \item if $\tilde g$ is a finite cover of $g$ then $\tilde g$ also verifies 1 and 2. 
\end{enumerate}

\begin{conj}
Let $X$ be a simply connected closed manifold or a simply connected finite simplicial complex and let $p\colon E\to X$ be a fiber bundle whose fiber $M$ is a closed manifold. Assume that $E$ admits a fiberwise irreducible partially hyperbolic diffeomorphism with low dimensional (1 or 2) center distribution. Then $p\colon E\to X$ is topologically trivial.
\end{conj}
\begin{q}
Can one construct an example of a simply connected manifold $X$ and a non-trivial fiber bundle $p\colon E\to X$ whose fiber $M$ is a closed manifold such that the total space $E$ can be equipped with a fiberwise irreducible partially hyperbolic diffeomorphism?
\end{q}

\begin{remark}
If one modifies the definition of ``fiberwise partially hyperbolic diffeomorphism'' to allow diffeomorphisms that permute the fibers (\ie factor over a non-trivial diffeomorphism of $X$) then the answer to the above question is positive~\cite{GORH}.
\end{remark}

\subsection{Acknowledgements} We thank Rafael de la Llave and Federico Rodriguez Hertz for very useful communications. Also we thank the referee for comments which helped to improve our presentation.

\end{document}